\documentclass[11pt]{article}
\usepackage{amsmath, amscd, amssymb, latexsym, epsfig, color, amsthm}
\usepackage{tikz, enumerate}
\usetikzlibrary{calc}
\usepackage{pstricks,graphicx,subfig}
\numberwithin{equation}{section}

\newtheorem{theorem}{Theorem}[section]
\newtheorem{proposition}[theorem]{Proposition}

\newtheorem{corollary}[theorem]{Corollary}

\newtheorem{lemma}[theorem]{Lemma}

\newtheorem{problem}[theorem]{Problem}

\theoremstyle{definition}
\newtheorem{definition}[theorem]{Definition}

\newtheorem{remark}[theorem]{Remark}

\DeclareMathOperator{\cs}{cs}
\DeclareMathOperator{\ncs}{ncs}
\DeclareMathOperator{\skel}{Skel}
\DeclareMathOperator\lk{\mathrm{lk}}
\DeclareMathOperator\st{\mathrm{st}}

\DeclareMathOperator{\conv}{\mathrm{conv}}
\DeclareMathOperator{\dist}{\mathrm{dist}}

\newcommand{\B}{\mathcal{B}}

\newcommand{\R}{{\mathbb R}}

\newcommand{\I}{{\mathcal I}}
\newcommand{\C}{{\mathcal C}}

\title{New families of highly neighborly centrally symmetric spheres}
\author{
	Isabella Novik\thanks{Research of IN is partially\textsl{} supported by NSF grants DMS-1664865 and DMS-1953815, and by Robert R.~\&  Elaine F.~Phelps Professorship in Mathematics. }\\
	\small Department of Mathematics\\[-0.8ex]
	\small University of Washington\\[-0.8ex]
	\small Seattle, WA 98195-4350, USA\\[-0.8ex]
	\small \texttt{novik@uw.edu}
	\and 
	Hailun Zheng\thanks{Research of HZ is partially supported by a postdoctoral fellowship from ERC grant 716424 - CASe.}\\
	\small Department of Mathematical Sciences\\[-0.8ex]
	\small University of Copenhagen\\[-0.8ex]
	\small Universitesparken 5, 2100 Copenhagen, Denmark \\[-0.8ex]
	\small \texttt{hz@math.ku.dk}
}
\begin{document}
\maketitle
\begin{abstract} 
	 In 1995, Josckusch constructed an infinite family of centrally symmetric (cs, for short) triangulations of $3$-spheres that are cs-$2$-neighborly. Recently, Novik and Zheng extended Jockusch's construction: for all $d$ and $n>d$, they constructed a cs triangulation of a $d$-sphere with $2n$ vertices, $\Delta^d_n$, that is cs-$\lceil d/2\rceil$-neighborly. Here, several new cs constructions, related to $\Delta^d_n$, are provided. It is shown that for all $k>2$ and a sufficiently large $n$, there is another cs triangulation of a $(2k-1)$-sphere with $2n$ vertices that is cs-$k$-neighborly, while for $k=2$ there are $\Omega(2^n)$ such pairwise non-isomorphic triangulations. It is also shown that for all $k>2$ and a sufficiently large $n$, there are $\Omega(2^n)$ pairwise non-isomorphic cs triangulations of a  $(2k-1)$-sphere with $2n$ vertices that are cs-$(k-1)$-neighborly. The constructions are based on studying facets of $\Delta^d_n$, and, in particular, on some necessary and some sufficient conditions similar in spirit to Gale's evenness condition. Along the way, it is proved that Jockusch's spheres $\Delta^3_n$ are shellable and an affirmative answer to Murai--Nevo's question about $2$-stacked shellable balls is given.
\end{abstract}
	
\section{Introduction} \label{sec:Intro}
In this paper, we construct new families of centrally symmetric (cs, for short) triangulations of spheres that are highly neighborly. Our constructions are based on studying the edge links and facets of the complex $\Delta^d_n$. Here, for odd $d\geq 3$, $\{\Delta^d_n : n>d\}$ is the only currently known infinite family of cs $d$-spheres that are cs-$\lceil d/2\rceil$-neighborly. In the process, we establish several properties of $\Delta^d_n$ that are natural cs analogs of the properties that the cyclic polytopes have.

A simplicial complex $\Delta$ is called $\ell$-neighborly if every $\ell$ of its vertices form a face. One famous example is $C(d+1,n)$ --- (the boundary complex of) the cyclic $(d+1)$-polytope with $n$ vertices. This object was  discovered and rediscovered by Carath\'eodory, Gale, and Motzkin \cite{Carath1911,Gale63,Motz57} among others, and it is $\lceil d/2\rceil$-neighborly. One reason triangulations of $d$-spheres that are $\lceil d/2\rceil$-neighborly are so sought after is the celebrated Upper Bound Theorem of McMullen \cite{McMullen70} (for polytopes) and Stanley \cite{Stanley75} (for spheres) asserting that among all triangulated $d$-spheres with $n$ vertices, any $\lceil d/2\rceil$-neighborly sphere simultaneously maximizes all the face numbers. These extremal properties may seem to suggest that  $\lceil d/2\rceil$-neighborly triangulations of $d$-spheres are extremely rare. Surprisingly, this intuition is quite wrong: such objects abound.
 Indeed, for $k\geq 2$, Shemer \cite{Shemer} constructed superexponentially many (in the number of vertices) combinatorial types of $k$-neighborly $2k$-polytopes; the current record lower bound on the number of combinatorial types of $k$-neighborly $2k$-polytopes is due to Padrol \cite{Padrol-13}. In fact, Kalai conjectured that for odd  $d\geq 3$ and a sufficiently large $n$, most of triangulated $d$-spheres with $n$ vertices are $\lceil d/2\rceil$-neighborly, \cite[Section 6.3]{Kal}.

For cs polytopes and spheres the situation with neighborliness is much more subtle. We say that a cs simplicial complex $\Delta$ is cs-$\ell$-neighborly if every set of $\ell$ of its vertices, no two of which are antipodal, forms a face of $\Delta$. While there do exist cs $(d+1)$-polytopes with $2(d+1)$ and $2(d+2)$ vertices that are cs-$\lceil d/2\rceil$-neighborly \cite{McMShep}, contrary to the non-cs situation, a cs $(d+1)$-polytope with more than $2(d+2)$ vertices cannot be cs-$\lceil d/2\rceil$-neighborly \cite{McMShep}, and a cs $(d+1)$-polytope with more than $2^{d+1}$ vertices cannot be even cs-$2$-neighborly \cite{LinNov}. On the other hand, an infinite family of cs triangulations of $3$-spheres that are cs-$2$-neighborly was constructed by Jockusch \cite{Jockusch95}. Furthermore, very recently the authors \cite{N-Z} extended Jockusch's result:  for all $d$ and $n>d$, they constructed a cs combinatorial $d$-sphere with $2n$ vertices, $\Delta^d_n$, that is cs-$\lceil d/2\rceil$-neighborly. (The family $\{\Delta^3_n : n\geq 4\}$ coincides with Jockusch's series of complexes.) It is worth pointing out that Lutz, see \cite[Chapter 4]{Lutz} and \cite{Manifold-page}, using a computer search, found several examples of highly neighborly cs spheres of dimensions $3$, $5$, and $7$ with few vertices. However, at present, the complexes  $\Delta^{2k-1}_n$ (for $k\geq2$) provide the {\em only} known up-to-date construction of a cs $(2k-1)$-sphere with an arbitrary even number of vertices that is cs-$k$-neighborly.\footnote{All Lutz's spheres possess vertex-transitive cyclic symmetry. As such, all of his $(2k-1)$-spheres with $2n>4k$ vertices are non-isomorphic to $\Delta^{2k-1}_n$.} The complex $\Delta^{2k}_{n}$ and the suspension of $\Delta^{2k-1}_{n-1}$ are the only two known constructions of cs $2k$-spheres that are cs-$k$-neighborly.

In this paper, we produce several new constructions. Our main results can be summarized as follows:
\begin{itemize}
\item For all $k\geq 2$ and a sufficiently large $n$, there exists a cs combinatorial $(2k-1)$-sphere with $2n$ vertices, $\Lambda^{2k-1}_n$, that is cs-$k$-neighborly and not isomorphic to $\Delta^{2k-1}_n$, see Theorems \ref{lm: edge link main result} and \ref{thm: distinct edge links}. Of course, the suspension of $\Lambda^{2k-1}_{n-1}$ provides us with an analogous result in even dimensions. 
\item For $k=2$, there exist $\Omega(2^n)$ pairwise non-isomorphic cs combinatorial $3$-spheres with $2n$ vertices that are cs-$2$-neighborly, see Theorem \ref{thm: 2^n combinatorial types}.
\item For all $k\geq 3$ and a sufficiently large $n$, there exist  $\Omega(2^n)$ pairwise non-isomorphic cs combinatorial $(2k-1)$-spheres with $2n$ vertices that are cs-$(k-1)$-neighborly, see Theorem \ref{thm:cs-(k-1)-neighb-constr}.
\end{itemize}

Many constructions in the non-cs world start with the cyclic polytope \cite{Kal, Shemer}. In the same spirit, all of our constructions start with  $\Delta^d_n$. For instance, the complex $\Lambda^{2k-1}_n$ is obtained as an edge link of $\Delta^{2k+1}_n$. This construction is inspired by the known fact that while an edge link of a $(k+1)$-neighborly $(2k+1)$-sphere is, in general, only $(k-1)$-neighborly, there exist edge links of $C(2k+2, n+2)$ that are isomorphic to $C(2k, n)$, and hence are $k$-neighborly. Also the complexes in the third construction are obtained by bistellar flips performed on $\Delta^{2k-1}_n$.  

The proofs of promised results require thorough understanding of complexes $\Delta^d_n$. Consequently, a big portion of the paper is devoted to establishing new properties of $\Delta^d_n$. For instance, we explicitly describe all facets of $\Delta^3_n$ as well as provide some sufficient and some necessary conditions on facets of $\Delta^{2k-1}_n$ for all $k>2$ (see Section \ref{sec:Facets}) that are similar in spirit to Gale's evenness condition on facets of cyclic polytopes \cite{Gale63}. One consequence of these results is that {\em all} Kalai's squeezed $(2k-1)$-balls with $n$ vertices \cite{Kal} are embeddable in $\Delta^{2k+1}_{2n+1}$ and also in $\Lambda^{2k-1}_{2n-1}$ as subcomplexes, see Proposition \ref{prop:squeezed-balls-as-subcomplexes}. We also prove that for $k\geq 2$ and a sufficiently large $n$, the complex $\Delta^{2k-1}_n$ admits only two automorphisms (i.e., the identity, and the involution that takes each vertex to its antipode), see Theorem \ref{thm: automorphism}. Along the way, we show that the spheres $\Delta^3_n$ are shellable, see Theorem \ref{thm:shellable}, and answer in the affirmative Murai--Nevo's question from \cite{MuraiNevo2013} about the existence of a $2$-stacked shellable ball whose boundary complex is not polytopal. It is our hope that the families of cs spheres $\Delta^{2k-1}_n$ and $\Lambda^{2k-1}_n$ will be a fruitful source for finding many {\em non-polytopal} constructions of spheres with additional interesting properties. We refer to \cite{GouveiaMacchiaWiebe,Pfeifle-20} for the first results in this direction.

The structure of the paper is as follows. In Section \ref{sec:Prelim}, after reviewing basics of simplicial complexes along with basics of combinatorial balls and spheres (see Section \ref{subsec:simplicial-compelxes}), we summarize the main definitions and results of \cite{N-Z} (see Section \ref{sec:Delta-and-B}). These include the definition of cs combinatorial spheres $\Delta^d_n$ and certain combinatorial balls $B^{d,i}_n$ as well as some of their properties. In Section \ref{sec:Facets} we study the facets of $\Delta^{2k-1}_n$ and more generally of $\partial B^{d,i}_n$. Section \ref{sec:two-constructions} is an intermission section: there we outline all of our high-dimensional constructions. In Section \ref{sec:k-neighb}, we discuss the edge links of $\Delta^{2k+1}_n$, and in particular, verify the promissed properties of $\Lambda^{2k-1}_n$. In Section \ref{sec:(k-1)-neighb} we construct many cs combinatorial $(2k-1)$-spheres that are cs-$(k-1)$-neighborly. Sections \ref{sec:many-3-spheres} and \ref{sec:shellable} are devoted to $3$-dimensional complexes: in Section \ref{sec:many-3-spheres}, we construct many cs combinatorial $3$-spheres that are cs-$2$-neighborly, while in Section \ref{sec:shellable} we prove shellability of $\Delta^3_n$ and answer Murai--Nevo's question. We close in Section \ref{sec:open problems} with a few open problems.

\section{Preliminaries} \label{sec:Prelim}
\subsection{Basics on simplicial complexes}  \label{subsec:simplicial-compelxes}
In this section we review several notions and results pertaining to simplicial complexes.
A \emph{simplicial complex} $\Delta$ on vertex set $V=V(\Delta)$ is a collection of subsets of $V$ that is closed under inclusion and contains all singletons: $\{v\}\in\Delta$ for all $v\in V$. The elements of $\Delta$ are called {\em faces}. The \emph{dimension of a face} $\tau\in\Delta$ is $\dim\tau:=|\tau|-1$. The \emph{dimension of $\Delta$}, $\dim\Delta$, is the maximum dimension of its faces. We refer to faces of dimension $0$ and $1$ as \emph{vertices} and \emph{edges}, respectively. To simplify the notation, for a face that is a vertex, we write $v$ instead of $\{v\}$.

A face of a simplicial complex $\Delta$ is a \textit{facet} if it is maximal w.r.t.~inclusion. We say that $\Delta$ is \emph{pure} if all facets of $\Delta$ have the same dimension; in this case, faces of codimension $1$ are called {\em ridges}. If $\Delta$ is pure, the \emph{facet-ridge graph} of $\Delta$ is the graph whose vertices are the facets of $\Delta$ and whose edges are pairs of facets that contain a common ridge.

Let $V$ be a finite set. Denote by $\overline{V}:=\{\tau \ : \ \tau\subseteq V\}$ the $(|V|-1)$-dimensional simplex (or $(|V|-1)$-simplex, for short) on vertex set $V$ and by $\partial \overline{V}:=\{\tau \ : \ \tau\subsetneq V\}$ the boundary complex of this simplex. For $v_1,\dots, v_n\in V$, we let $(v_1,v_2,\dots, v_n)$ be a path with edges $\{v_1,v_2\}, \{v_2,v_3\},\dots, \{v_{n-1},v_n\}$ if $v_n\neq v_1$, or a cycle if $v_n=v_1$. In particular, a path $(v_1, v_2)$ of length one is a $1$-simplex, so it can also be written as $\overline{\{v_1,v_2\}}$. 

Let $\Delta$ be a simplicial complex. If $W\subseteq V(\Delta)$ is any subset of vertices, we define the {\em restriction} of $\Delta$ to $W$ to be the subcomplex $\Delta[W]=\{F\in \Delta: F\subseteq W\}$. The \emph{$k$-skeleton} of $\Delta$, $\skel_k(\Delta)$, is the subcomplex of $\Delta$ consisting of all faces of dimension $\leq k$. If $\tau$ is a face of $\Delta$, then the {\em star} and the {\em link of $\tau$ in $\Delta$} are the following subcomplexes of $\Delta$: 
\[\st(\tau, \Delta)=\{\sigma\in \Delta \ : \ \sigma\cup\tau \in \Delta\} \mbox{ and }
\lk(\tau,\Delta):= \{\sigma\in \st(\tau, \Delta) \ : \ \sigma\cap\tau=\emptyset\}.
\]
When the context is clear, we sometimes write $\st(\tau)$ and $\lk(\tau)$ in place of $\st(\tau, \Delta)$ and $\lk(\tau, \Delta)$.

If $\Delta$ is pure and $\Gamma$ is a full-dimensional pure subcomplex of $\Delta$, then $\Delta\backslash \Gamma$ is the subcomplex of $\Delta$ generated by those facets of $\Delta$ that are not in $\Gamma$. If $\Delta$ and $\Gamma$ are simplicial complexes on disjoint vertex sets, then the \textit{join} of $\Delta$ and $\Gamma$ is the simplicial complex $\Delta*\Gamma = \{\sigma \cup \tau \ : \ \sigma \in \Delta \text{ and } \tau \in \Gamma\}$. Two important special cases are the \emph{cone} over $\Delta$ with apex $v$ defined as  the join $\Delta\ast\overline{v}$ and the \emph{suspension} of~$\Delta$, $\Sigma\Delta$, defined as the join of $\Delta$ with a $0$-dimensional sphere. In the rest of the paper, we write $\Delta \ast v$ in place of  $\Delta\ast\overline{v}$. 

We will be mainly focusing on the following two classes of simplicial complexes. A \emph{combinatorial $d$-ball} is a simplicial complex PL homeomorphic to a $d$-simplex. Similarly, a \emph{combinatorial $(d-1)$-sphere} is a simplicial complex PL homeomorphic to the boundary complex of a $d$-simplex. The link of any face in a combinatorial sphere is a combinatorial sphere. On the other hand, the link of a face $\tau$ in a combinatorial $d$-ball $B$ is either a combinatorial ball or a combinatorial sphere; in the former case we say that $\tau$ is  a \emph{boundary face} of $B$, and in the latter case that $\tau$ is an \emph{interior face} of $B$. The \emph{boundary complex} of $B$, $\partial B$, is the subcomplex of $B$ that consists of all boundary faces of $B$; in particular, $\partial B$ is a combinatorial $(d-1)$-sphere. (See \cite{Hudson} for additional background on PL topology.)

Let $\Delta$ be a pure simplicial complex, and assume that
\[A\in \Delta, \, B\notin \Delta, \, \lk(A, \Delta)=\partial \overline{B}.\]
The process of replacing the subcomplex $\st(A, \Delta)=A*\partial\overline{B}$ with $\partial \overline{A}* B$ is called a {\em bistellar flip}.  Two complexes are called {\em bistellar equivalent} if one can be obtained from the other through a sequence of bistellar flips. It is clear from this definition that bistellar equivalent complexes are PL homeomorphic. A much more surprising result that emphasizes the significance of bistellar flips is the following theorem of Pachner:

\begin{theorem}[\cite{Pachner}] A simplicial complex $\Delta$ is a combinatorial $(d-1)$-sphere if and only if $\Delta$ is bistellar equivalent to the boundary complex of a $d$-simplex.
\end{theorem}

A simplicial complex $\Delta$ is \textit{centrally symmetric} or {\em cs} if its vertex set is endowed with a {\em free involution} $\alpha: V(\Delta) \rightarrow V(\Delta)$ that induces a free involution on the set of all nonempty faces of $\Delta$. In more detail,  for all nonempty faces $\tau \in \Delta$, the following holds: $\alpha(\tau)\in\Delta$, $\alpha(\tau)\neq \tau$, and $\alpha(\alpha(\tau))=\tau$. To simplify notation, we write $\alpha(\tau)=-\tau$ and refer to $\tau$ and $-\tau$ as {\em antipodal faces} of $\Delta$. Similarly, if $\Gamma$ is a subcomplex of $\Delta$ we write $-\Gamma$ in place of $\alpha(\Gamma)$.

One example of a cs combinatorial $d$-sphere is the boundary complex of the $(d+1)$-dimensional \emph{cross-polytope}, $\partial\C^*_{d+1}$. The polytope $\C_{d+1}^*$ is the convex hull of $\{\pm e_1,\pm e_2,\ldots, \pm e_{d+1}\}$, where $e_1,e_2,\ldots, e_{d+1}$ are the endpoints of the standard basis in $\R^{d+1}$. As an abstract simplicial complex, $\partial\C^*_{d+1}$ is the $(d+1)$-fold suspension of $\{\emptyset\}$. Equivalently, it is the collection of all subsets of $V_{d+1}:=\{\pm 1,\ldots,\pm(d+1)\}$ that contain at most one vertex from each pair $\{\pm i\}$. In particular, every cs simplicial complex on vertex set $V_n$ is a subcomplex of $\partial\C^*_n$.  

Let $\Delta\subseteq \partial\C^*_{n}$ be a simplicial complex, possibly non-cs, and let $1\leq i\leq n$. We say that $\Delta$ is \emph{cs-$i$-neighborly} (w.r.t.~$V_n$), if $\skel_{i-1}(\Delta)=\skel_{i-1}(\partial \C_n^*)$. For $i=1$, this simply means that $V(\Delta)=V_n$. Furthermore, we say that $\Delta$ is {\em exactly cs-$i$-neighborly (w.r.t.~$V_n$) if $\Delta$ is cs-$i$-neighborly but not cs-$(i+1)$-neighborly (w.r.t.~$V_n$).} For convenience, we also refer to  simplices (i.e., faces of $\partial\C^*_{n}$) as cs-$0$-neighborly complexes. 

For a $d$-dimensional simplicial complex $\Delta$, we let $f_i = f_i(\Delta)$ be the number of $i$-dimensional faces of $\Delta$ for $-1\leq i\leq d$. The vector $(f_{-1}=1, f_0, \ldots, f_{d})$ is called the $f$\emph{-vector} of $\Delta$. The \emph{$h$-vector} of $\Delta$, $(h_0,h_1, \dots, h_{d+1})$, is defined by the relation $$\sum_{i=0}^{d+1} h_i t^{d+1-i}=\sum_{i=0}^{d+1} f_{i-1}(t-1)^{d+1-i}.$$

If $\Delta$ is a combinatorial $d$-sphere, then the Dehn-Sommerville relations \cite{klee64} assert that $h_i=h_{d+1-i}$ for all $0\leq i\leq d+1$. This implies the following useful approximation:
\begin{lemma}\label{lm: facet count}
Let $k\geq 1$ be a fixed integer and let $\Delta\subseteq \partial \C^*_n$ be a combinatorial $(2k-1)$-sphere on $V_n$. If $\Delta$ is cs-$k$-neighborly, then $\Delta$ has $2^k\binom{n}{k}+O(n^{k-1})$ facets.
\end{lemma}
\begin{proof}
	Since $\Delta$ and $\partial C^*_n$ have the same $(k-1)$-skeleton, it follows that for all $i\leq k$, $f_{i-1}(\Delta)=f_{i-1}(\partial C^*_n)=2^i\binom{n}{i}$. Consequently, $h_i(\Delta)=2^i\binom{n}{i} + O(n^{i-1})$ for all $i\leq k$, and we infer from the Dehn-Sommerville relations that
	$f_{2k-1}(\Delta)=\sum_{i=0}^{2k} h_i(\Delta)=h_k(\Delta)+2\big(h_0(\Delta)+h_1(\Delta)+\dots h_{k-1}(\Delta)\big)=2^k\binom{n}{k}+O(n^{k-1}). $
\end{proof}

The following notion takes its origins in the Generalized Lower Bound Theorem \cite{McMullenWalkup71,MuraiNevo2013,Stanley80}. 
A combinatorial $d$-ball $B$ is called \emph{$i$-stacked} (for some $0\leq i\leq d$), if all interior faces of $B$ are of dimension $\geq d-i$, that is, $\skel_{d-i-1}(B)=\skel_{d-i-1}(\partial B)$, and it is called {\em exactly $i$-stacked} if, in addition, $B$ has an interior $(d-i)$-face. For instance, a ball is $0$-stacked or exactly $0$-stacked if and only if it is a simplex. A ball is $1$-stacked if its facet-ridge graph is a tree; furthermore, it is exactly $1$-stacked if it is not a simplex. ($1$-stacked balls are also known in the literature as stacked balls.) A combinatorial $(d-1)$-sphere is called \emph{$i$-stacked} if it is the boundary complex of some $i$-stacked combinatorial $d$-ball.

We close this subsection with three lemmas that will be used in our main constructions. For the first and third, see \cite[Lem.~2.2 and 2.3]{N-Z}.
\begin{lemma}\label{lm: stackedness}
	Let $B_1$ and $B_2$ be combinatorial balls of dimension $d_1$ and $d_2$, respectively.
	If $B_1$ is $i_1$-stacked and $B_2$ is $i_2$-stacked, then
	\begin{enumerate}
		\item The complex $B_1*B_2$ is an $(i_1+i_2)$-stacked combinatorial $(d_1+d_2+1)$-ball. 
		\item If $d_1=d_2=d$, $i_1\leq i_2$, and  $B_1\cap B_2\subseteq \partial B_1\cap \partial B_2$ is a combinatorial $(d-1)$-ball that is $i_3$-stacked for some $i_3< i_2$, then $B_1\cup B_2$ is an $i_2$-stacked combinatorial $d$-ball.
	\end{enumerate}
\end{lemma}
\begin{lemma}\label{lm: exact stackedness}
	Assume that $B_1, B_2$ and $B_1\cup B_2$ are combinatorial $d$-balls and that $B_1\cap B_2\subseteq \partial B_1\cap \partial B_2$ is a combinatorial $(d-1)$-ball. If $B_1\cup B_2$ is $i$-stacked, then both $B_1$ and $B_2$ are $i$-stacked while $B_1\cap B_2$ is $(i-1)$-stacked. Furthermore, if $B_1\cup B_2$ is exactly $i$-stacked, then either one of $B_1$ and $B_2$ is exactly $i$-stacked, or $B_1\cap B_2$ is exactly $(i-1)$-stacked.
\end{lemma}
\begin{proof}
	If $B_1\cup B_2$ is $i$-stacked, then all interior faces of $B_1\cup B_2$ are of dimension $\geq d-i$. Since every interior face of $B_1$, $B_2$ or $B_1\cap B_2$ is an interior face of $B_1\cup B_2$, it follows from the definition of stackedness that both $B_1$ and $B_2$ are $i$-stacked while $B_1\cap B_2$ is $(i-1)$-stacked. If $B_1\cup B_2$ is exactly $i$-stacked, then $B_1\cup B_2$ has an interior $(d-i)$-face $F$. This face $F$ must be an interior face of one of the complexes $B_1$, $B_2$ and $B_1\cap B_2$. Hence the second statement also holds.
\end{proof}

\begin{lemma}\label{lm: complement}
	Let $k\geq 1$ be an integer. Let $\Delta\subseteq \partial \C^*_n$ be a combinatorial $(2k-1)$-sphere that is cs-$k$-neighborly w.r.t.~$V_n$, and let $B\subseteq \Delta $ be a combinatorial $(2k-1)$-ball that is both cs-$(k-1)$-neighborly w.r.t.~$V_n$ and $(k-1)$-stacked. Then $\Delta \backslash B$ is a combinatorial $(2k-1)$-ball that is cs-$k$-neighborly and $k$-stacked.
\end{lemma}

\subsection{The complexes $\Delta^{d}_n$ and $B^{d, i}_n$} \label{sec:Delta-and-B}
In \cite{N-Z} (building on Jockusch's construction from \cite{Jockusch95}), for each $d\geq 1$ and $n\geq d+1$, we constructed a cs combinatorial $d$-sphere $\Delta^d_n$ on $V_n$ that is cs-$\lceil \frac{d}{2}\rceil$-neighborly. Below we briefly review this construction and discuss some of the properties that the complexes $\Delta^d_n$ possess. An essential part of the construction is a family of cs-$i$-neighborly and $i$-stacked balls $B^{d,i}_n$.
\begin{definition}\label{def}
	Let $d\geq 1$, $i\leq \lceil d/2\rceil$, and $n\geq d+1$ be integers. Define $\Delta^{d}_n$ and $B^{d,i}_n$ inductively as follows: 
	\begin{itemize}
		\item For the initial cases, define
		$\Delta^{1}_{n}:=(1, 2,\dots, n,-1,-2,\dots, -n,1)$, $\Delta^{d}_{d+1}:=\partial \C^*_{d+1}$, 
		$B^{d,j}_n:=\emptyset$   if $j<0$, and $B^{1, 0}_n:=(-1, n)$. (In particular, $B^{1,j}_n\subseteq \Delta^{1,1}_n$ for all $j\leq 0$.)
		
		\item
		If $\Delta^{d-1}_{m}$ and $B^{d-1,i}_m \subseteq \Delta^{d-1}_{m}$ are already defined for all $i\leq \lfloor (d-1)/2\rfloor$ and $m\geq d$, then for $d=2k$ and $n\geq 2k$, define $B^{2k-1,k}_{n}:=\Delta^{2k-1}_{n}\backslash B^{2k-1,k-1}_{n}$; furthermore, for all $n\geq d+1$ and $i\leq \lfloor d/2\rfloor$, define $$B^{d, i}_n:=\left(B^{d-1, i}_{n-1}*n\right) \cup\left((-B^{d-1,i-1}_{n-1})*(-n)\right).$$ 
		
		\item	 
		If $\Delta^{d}_n$ is already defined, then define $\Delta^{d}_{n+1}$ as the complex obtained from $\Delta^{d}_n$ by replacing the subcomplex $B^{d,\lceil d/2\rceil -1}_n$ with $\partial B^{d,\lceil d/2\rceil -1}_n * (n+1)$ and $-B^{d,\lceil d/2\rceil -1}_n$ with $\partial (-B^{d,\lceil d/2\rceil -1}_n) * (-n-1)$.
	\end{itemize}
\end{definition}

\noindent Note that in \cite{N-Z} the vertex set of $\Delta^{d}_n$ and $B^{d, i}_n$ is $\{\pm v_1, \pm v_2, \dots, \pm v_n\}$ while in Definition \ref{def} the vertex set is $\{\pm 1, \pm 2, \dots, \pm n\}$. 

Putting for the moment the question of whether the objects $\Delta^d_n$ and $B^{d,i}_n$ are well-defined aside, observe that Definition \ref{def} and induction on $n$ imply that for all $d\geq 1$,  $B^{d, 0}_n$ is the simplex on the vertex set $\{-1, n-d+1, n-d+2, \dots, n\}$. Another consequence of  Definition \ref{def} is that
for $d\geq 2$ and $i\leq \lceil d/2\rceil -1$,
\begin{equation} \label{B^d-in-terms of-B^{d-2}}
B^{d, i}_n=\left( B^{d-2, i}_{n-2}*(n-1, n) \right)  \cup  \left( (-B^{d-2, i-1}_{n-2})* (n, -n+1, -n)\right)\cup\left( B^{d-2, i-2}_{n-2}*(n-1, -n)\right).
\end{equation}
In particular, letting $d=3$ and $i=1$ and using definitions of $B^{1,1}_{n-2}$ and $B^{1,0}_{n-2}$, we obtain that
\begin{equation} \label{eq:facets-of-B^{3,1}}
B^{3,1}_n=\Big( (n-2, n-3, \dots, 1, -n+2, -n+3, \dots, -1)* (n-1, n) \Big) \cup \Big( (1, -n+2)*(n, -n+1, -n)\Big) .
\end{equation}
In Section \ref{sec:Facets}, we will use this description of $B^{3,1}_n$ to characterize all facets of $\Delta^3_n$.

A big portion of \cite[Section 3]{N-Z} is devoted to showing that the objects $\Delta^d_n$ and $B^{d,i}_n$ are well-defined (including the fact that $\Delta^{d}_n\supseteq B^{d,\lceil d/2\rceil -1}_n$). The proof relies on a few crucial properties of $B^{d,i}_n$, see \cite[Lem.~3.3, 3.4, 3.6 and Cor.~3.7]{N-Z}, summarized in the following lemma. 
\begin{lemma}\label{lm: prop of B^{d, i}_n}
    Let $d\geq 2$ and $n\geq d+1$. Then for all $0\leq i\leq j\leq \lfloor d/2 \rfloor$ and $k\leq \lceil d/2 \rceil$, 
    \begin{enumerate}
		\item $B^{d,i}_n$ is a combinatorial $d$-ball that is cs-$i$-neighborly (w.r.t.~$V_n$) and $i$-stacked; furthermore, $B^{d,i}_n$ shares no common facets with $-B^{d,i}_n$;
			\item $B^{d, k-1}_n\subseteq -B^{d, k}_n$;
    	\item $\partial B^{d, i}_n=\left( \partial B^{d-1, i}_{n-1}*n\right)   \cup \left( \partial (-B^{d-1, i-1}_{n-1})*(-n)\right) \cup \left( B^{d-1, i}_{n-1}\backslash -B^{d-1, i-1}_{n-1}\right)$;
    	\item $B^{d-1, i}_n\subseteq \partial B^{d, j}_n$;
			\item $B^{d, \lceil d/2\rceil-1}_{n+1} \cup -B^{d, \lceil d/2\rceil-1}_{n+1}\subseteq\left( \partial B^{d, \lceil d/2 \rceil -1}_{n} *(n+1) \right) \cup \left( \partial \big(-B^{d, \lceil d/2 \rceil -1}_{n}\big) *(-n-1) \right)$.
    \end{enumerate} 
\end{lemma}

The main result of \cite{N-Z}, see \cite[Theorem 3.8]{N-Z}, is 
\begin{theorem}\label{thm: Delta^{d}_n}
	For all $d\geq 2$ and $n\geq d+1$, the complex $\Delta^d_n$ is well-defined. It is a cs combinatorial $d$-sphere with vertex set $V_n$ that is cs-$\lceil d/2\rceil$-neighborly. 
\end{theorem}

The proof of Theorem \ref{thm: Delta^{d}_n} utilizes Lemma \ref{lm: prop of B^{d, i}_n} along with the following inductive method of constructing cs combinatorial $d$-spheres that are cs-$i$-neighborly, see \cite[Lemma 3.1]{N-Z}. We will use this method, which can be considered a cs analog of Shemer's sewing technique, in Section \ref{sec:many-3-spheres} to construct many cs combinatorial $3$-spheres that are cs-$2$-neighborly.
\begin{lemma}\label{lm: induction method}
	Let $d\geq 1$ and $1\leq i\leq \lceil d/2\rceil$ be integers. Assume that $\Gamma$ is a cs combinatorial $d$-sphere with $V(\Gamma)=V_n$ that is cs-$i$-neighborly. Assume further that $B\subseteq\Gamma$ is a combinatorial $d$-ball that satisfies the following properties: 
	\begin{itemize}
		\item the ball $B$ is both cs-$(i-1)$-neighborly w.r.t.~$V_n$ and $(i-1)$-stacked, and
		\item the balls $B$ and $-B$ share no common facets.
	\end{itemize} 
	Then the complex $\Gamma'$ obtained from $\Gamma$ by replacing $B$ with $\partial B*(n+1)$ and $-B$ with $\partial(-B)*(-n-1)$ is a cs combinatorial $d$-sphere with $V(\Gamma')=V_{n+1}$ that is cs-$i$-neighborly.
\end{lemma}

We will also need the following strengthening of Lemma \ref{lm: prop of B^{d, i}_n}(1). It follows easily from the definition of $B^{d, i}_n$, Lemmas \ref{lm: exact stackedness} and \ref{lm: complement}, and Theorem \ref{thm: Delta^{d}_n}. We leave it as an exercise to the reader. 
\begin{lemma}\label{remark on B^{d, i}_n}
	For $d\geq 2$ and $i\leq \lceil \frac{d}{2}\rceil$, $B^{d, i}_n$ is exactly cs-$i$-neighborly (w.r.t.~$V_n$) and exactly $i$-stacked. 
\end{lemma}

To close this section we mention two additional properties of $\Delta^d_n$ that will be handy. The first one is \cite[Cor.~3.5 and Prop.~4.1]{N-Z}. The second one was proved in \cite[Prop.~4.4]{N-Z} (by a different method) for the case of an odd $d$, but it is new for even values of $d$. 

\begin{lemma}\label{lm: prop of Delta^{2k-1}_n}
	Let $k\geq 2$ and $n\geq 2k$. Then $\Delta^{2k-1}_n=\partial B^{2k,k}_n$; in particular, the sphere $\Delta^{2k-1}_n$ is $k$-stacked. Moreover, $$\lk(\{n-1, n\}, \Delta^{2k-1}_n)=\Delta^{2k-3}_{n-2}.$$
\end{lemma}

\begin{proposition}\label{prop: odd even sphere relation}
	For any $d\geq 1$ and $n\geq d+2$, the complex $\Delta^d_n$ is a subcomplex of $\Delta^{d+1}_n$.
\end{proposition}
 
\begin{proof}
	The proof is by induction on $n$. The claim holds for $n=d+2$. Indeed, $\Delta^{d+1}_{d+2}=\partial \C^*_{d+2}$, and so it contains as subcomplexes \emph{all} cs complexes on $V_{d+2}$. For $n>d+2$, notice that $$\partial B^{d, \lfloor \frac{d-1}{2}\rfloor}_n \subseteq B^{d, \lfloor \frac{d-1}{2}\rfloor}_n \stackrel{(\diamond)}{\subseteq} \partial B^{d+1, \lfloor \frac{d}{2}\rfloor}_n,\mbox{ and hence } \pm \big(\partial B^{d, \lfloor \frac{d-1}{2}\rfloor}_n*(n+1)\big)\subseteq \pm \big(\partial B^{d+1, \lfloor \frac{d}{2}\rfloor}_n*(n+1)\big).$$ Here the inclusion $(\diamond)$ follows from Lemma \ref{lm: prop of B^{d, i}_n}(4). 
	Since $\Delta^i_{n+1}$ (for $i=d,d+1$) is obtained from $\Delta^{i}_n$ by replacing  $\pm B^{i, \lfloor \frac{i-1}{2}\rfloor}_n$ with $\pm \big(\partial B^{i, \lfloor \frac{i-1}{2}\rfloor}_n*(n+1)\big)$, that is, since all new faces belong to $\pm \big(\partial B^{i, \lfloor \frac{i-1}{2}\rfloor}_n*(n+1)\big)$, the claim follows from the inductive hypothesis on $n$.
\end{proof}

\section{The facets of $\Delta^{2k-1}_n$} \label{sec:Facets}
All new constructions in this paper require good understanding of facets of $\Delta^{2k-1}_n$. With this in mind, we start this section with a complete characterization of facets of $\Delta^{3}_n$. We then discuss the question of which subsets of $V_n$ can be facets of $\Delta^{2k-1}_n$ for $k\ge 3$, and provide certain sufficient and certain necessary conditions.

Our main tool is the following decomposition of $\Delta^{2k-1}_n$ into pure subcomplexes that are pairwise facet-disjoint: for all $n\geq 2k$,
\begin{equation}  \label{eq:decomposition}
\Delta^{2k-1}_n=\left( \Delta^{2k-1}_{2k}\backslash \pm B^{2k-1, k-1}_{2k}\right) \cup \left( \cup_{s=2k+1}^{n} \pm (\partial B^{2k-1, k-1}_{s-1} * s) \backslash \pm B^{2k-1, k-1}_s\right)  \cup \left( \pm B^{2k-1, k-1}_n\right).
\end{equation}
This decomposition is an immediate consequence of Definition \ref{def} and Lemma \ref{lm: prop of B^{d, i}_n}(5).

\subsection{Dimension three}
Combining equations \eqref{eq:facets-of-B^{3,1}} and \eqref{eq:decomposition} and recalling that $\Delta^3_4=\partial\C^*_4$ provides the following explicit description of facets of $\Delta^3_n$:
\begin{lemma}\label{lm: facets}
Let $n\geq 4$. The collection of facets of $\Delta^{3}_n$ consists of
\begin{enumerate}
	\item The facets of $\pm B^{3,1}_n$. They are the following sets along with their antipodes:
\begin{eqnarray*}	\{i,i+1, n-1,n\},\;\;\{-i,-i-1,n-1,n\} \quad \mathrm{for}\;\;1\leq i \leq n-3, \\
	 \{1,-n+2, n-1, n\},\;\; \{1,-n+2,-n+1, n\},\;\; \{1,-n+2,-n+1,-n\}. 
	\end{eqnarray*}
	\item The facets of $\bigcup_{s=5}^n \pm (\partial B^{3,1}_{s-1}*s)\backslash \pm B^{3,1}_s$. They are the following sets along with their antipodes:
	\begin{eqnarray*}\{i,i+1, \ell, \ell+2\}, \;\;\{-i, -i-1,\ell, \ell+2\},\;\; \{1, -\ell+1, \ell, \ell+2\}  &\mathrm{ for }& 1\leq i, i+1 <\ell\leq n-2; \\
	\{\ell,\ell+1,\ell+2,-\ell-3\},\;\;\{-1,\ell,\ell+2,-\ell -3\} &\mathrm{ for }& 2\leq \ell \leq n-3.
	\end{eqnarray*}
	\item The facets of $\Delta^{3}_4\backslash \pm B^{3,1}_4$. They are $\{1,2,-3,4\}, \{1,2,3,-4\}, \{1,-2,3,-4\}$ and their antipodes.
\end{enumerate}
\end{lemma}

Lemma \ref{lm: facets} allows us to compute links of edges of $\Delta^3_n$ and their sizes. This (rather technical) information will be used in Sections \ref{sec:k-neighb} and \ref{sec:many-3-spheres} to provide constructions of cs spheres that are highly cs-neighborly as well as to show that many of them are non-isomorphic; it will also be used in Section \ref{sec:(k-1)-neighb} to prove that for a sufficiently large $n$, the sphere $\Delta^3_n$ admits only two automorphisms. The links of edges in a $3$-sphere are (graph-theoretic) cycles, hence we can talk about their length.

\begin{corollary}\label{cor: edge links}
 Let $n\geq 6$ and let $e$ be an edge of $\Delta^{3}_n$.
Then
	$$f_0\big(\lk(e, \Delta^{3}_n)\big) \mbox{ is } \begin{cases}
	2n-4 & e=\pm \{1,2\},\; \pm \{n-1,n\}\\
	2(n-i)-1 & e=\pm \{i,i+1\},\; 2\leq i\leq n-3\\
	2n-5 & e=\pm \{n-2,n\}\\
	2\ell+1 & e=\pm\{\ell, \ell+2\},\; 3\leq \ell\leq n-3\\
	\leq 6 & \mathrm{otherwise}.
	\end{cases}$$
Furthermore, the link $\lk(\{1,2\},\Delta^3_n)$ is a cs cycle of length $2n-4$ that contains all pairs of the form $\{i, i+2\}$, for $3\leq i\leq n-2$, as edges. Similarly, for $3\leq \ell \leq n-2$, the link $\lk(\{\ell, \ell+2\}, \Delta^{3}_n)$ contains the path $(2, 1, -\ell+1, -\ell+2, \dots, -2, -1)$ as a subcomplex.
\end{corollary}

\begin{proof}
We use Lemma \ref{lm: facets}. For any $1\leq i\leq n-3$, the edges in the link of $\{i, i+1\}$ are those of the form $\pm \{\ell, \ell+2\}$ for $i+2\leq \ell\leq n-2$, $\pm \{n-1,n\}$, $\{i+2, -i-3\}$, along with $\{i-1, -i-2\}, \{i-1, i+3\}$ if $i\geq 2$ and with $\{-3,4\}$ if $i=1$.  
Together with the fact that $\lk(\{n-1, n\},\Delta^{3}_n)=\Delta^{1}_{n-2}$ this completes the proof of the first two cases and verifies the statement about the link of $\{1,2\}$ in the ``furthermore'' part.

Similarly, for $3\leq \ell \leq n-3$, the edges in the link of $\{\ell, \ell+2\}$ are those of the form  $\pm \{i, i+1\}$ for $1\leq i\leq \ell-2$,  $\{1, -\ell+1\}$, $\{-1, -\ell-3\}$, $\{\ell+1, -\ell-3\}$, together with the path $(\ell-1, \ell+4, \ell+1)$ if $\ell\leq n-4$ and with the path $(n-4,n, n-2)$ if $\ell=n-3$. In the same vein, the link of $\{n-2, n\}$ consists of the edges $\pm \{i, i+1\}$ for $1\leq i\leq n-4$,  $\{1, -n+3\}$, and the path $(-1, n-1, n-3)$.  This completes the proof of the third and fourth cases and also of the ``furthermore'' part.

Finally, observe that for $n\geq 6$, $\lk(\{n-2,n-1\})=(-1, -n, n-3, n, -1)$, $\lk(\{1,3\})=(-2,-4,2,5,-2)$, and $\lk(\{2,4\})=(-1, -3, 1, 6, 3, -5,-1)$, and that it follows from Lemma \ref{lm: facets} that the links of all other edges have at most six vertices. 
\end{proof}

\subsection{Higher dimensions}
While at present we do not have a complete description of facets of $\Delta^{2k-1}_n$ for $k>2$, we devote this section to establishing certain necessary as well as certain sufficient conditions for a subset of $V_n$ to be a facet. In particular, we identify a big chunk of facets of $\Delta^{2k-1}_n$. 

Let $M: \R\to \R^d$, $t \mapsto (t, t^2, \dots, t^d)$, be the {\em moment curve} in $\R^d$. The {\em cyclic polytope} $C(d, n)$ is the convex hull of $n$ distinct points on this curve, that is,
\[C(d, n):=\conv\big(M(t_1), M(t_2), \dots, M(t_n)\big) \mbox{ for some $t_1<t_2<\dots <t_n$}.\]
In \cite{Gale63} Gale proposed a criterion that characterizes (the vertex sets of the) facets of $C(d, n)$. Denote the vertex set of $C(d,n)$ by $[n]:=\{1,2,\dots,n\}$  where we identify $M(t_i)$ with $i$. A $d$-subset $F$ forms a facet of $C(d,n)$ if and only if the following ``evenness" condition is satisfied: if $i <j$ are not in $F$, then the number of elements in $F$ between $i$ and $j$ is even. In particular, Gale's evenness condition implies that the combinatorial type of $C(d,n)$ does not depend on the specific choice of points on the moment curve. It also implies that a ``typical" facet of $C(2k, n)$, written in the increasing order of elements, is of the form $\{i_1,i_1+1, i_2,i_2+1,\dots, i_k,i_{k}+1\}$. (We refer the reader to books \cite{Gru-book,Ziegler} for more background on cyclic polytopes and on polytopes in general.)

The following two lemmas provide necessary conditions on facets of $\partial B^{d, i}_n$, and in particular on facets of $\Delta^{2k-1}_n$ (see Lemma \ref{lm: facet}); the latter result is similar in spirit to Gale's evenness condition.

\begin{lemma}\label{lm: facet inclusion}
	Let $d\geq 2$, $n\geq d$, $j\leq d/2$ be integers, and let $F=\{p_1, \ldots, p_d\}$ be a facet of $\partial B^{d, j}_n$, where $|p_1|<|p_2|<\cdots < |p_d|$. Then for any $1\leq i\leq d$, $\{p_1, \ldots, p_i\}$ is a facet of $\pm \partial B^{i, j'}_{n'}$ for some $j'\leq \min\{i/2, j\}$ and $n'\leq n$. In particular, $\{p_1,p_2\}$ is a facet of $\Delta^1_{n''}$ for some $n''\leq n$.
\end{lemma}
\begin{proof}
	It suffices to treat the case of  $i=d-1$. Let $G=F\backslash\{p_d\}$. If $j<d/2$, then by Lemma~\ref{lm: prop of B^{d, i}_n}(3),

	\[\partial B^{d, j}_n=\left( \partial B^{d-1, j}_{n-1}*n\right) \cup \left( \partial (-B^{d-1, j-1}_{n-1})*(-n)\right) \cup  \left( B^{d-1, j}_{n-1}\backslash -B^{d-1, j-1}_{n-1}\right).\]
	Hence either $G\in \partial B^{d-1, j}_{n-1}\cup \partial (-B^{d-1, j-1}_{n-1})$, in which case we are done, or $F\in B^{d-1, j}_{n-1}$, and so $G\in B^{d-2, j}_{n-2}\cup (-B^{d-2,j-1}_{n-2})$. In the latter case, by applying Lemma \ref{lm: prop of B^{d, i}_n}(4) if $d\geq 2j+2$ or by using that $B^{2j-1, j}_{n-2}\subseteq \Delta^{2j-1}_{n-2}=\partial B^{2j, j}_{n-2}$ if $d=2j+1$, we infer that $G\in \partial B^{d-1, j}_{n-2}\cup \partial (-B^{d-1,j-1}_{n-2})$.
	
	Otherwise if $d=2j$, then $\partial B^{2j, j}_n=\Delta^{2j-1}_n$, and so by equation \eqref{eq:decomposition},
	$$\partial B^{2j, j}_n=\left( \Delta^{2j-1}_{2j}\backslash \pm B^{2j-1, j-1}_{2j}\right) \cup \left( \cup_{i=2j+1}^{n} \pm (\partial B^{2j-1, j-1}_{i-1} * i) \backslash \pm B^{2j-1, j-1}_i\right)  \cup \left( \pm B^{2j-1, j-1}_n\right).$$
    If $F\in \Delta^{2j-1}_{2j}=\partial\C^*_{2j}$, then $G\in \partial \C^*_{2j-1}=\big(\Delta^{2j-3}_{2j-2}*(2j-1)\big) \cup \big(\Delta^{2j-3}_{2j-2}*(-2j+1)\big)$, and so
	$G\in \pm\left( \partial B^{2j-2, j-1}_{2j-2}* (2j-1)\right) \subseteq \pm \partial B^{2j-1, j-1}_{2j-1}$, where the last step is by Lemma \ref{lm: prop of B^{d, i}_n}(3). The case of $F\in \pm B^{2j-1, j-1}_n$ was already treated in the previous paragraph.
	Finally, in the case when $F$ belongs to the middle term, the result obviously holds.
\end{proof}

\begin{lemma}\label{lm: facet}
Let $d\geq 2$, $0\leq i\leq d/2$, and $n\geq d$ be integers. Let $F=\{p_1, p_2, \dots, p_{d}\}$ be a facet of $\partial B^{d, i}_n$, where $|p_1|<|p_2|< \dots < |p_{d}|$. Then 
	\begin{enumerate}
		\item $|p_{2s}|-|p_{2s-1}|\leq 2$ for all $2\leq s\leq d/2$, and 
		\item $|p_2|-|p_1|=1$ unless $|p_1|=1$.
	\end{enumerate}
	In particular, since $\Delta^{2k-1}_n=\partial B^{2k,k}_n$, the facets of $\Delta^{2k-1}_n$ satisfy these conditions.
\end{lemma}

	\begin{proof}
	If $d=2k+1$, then the statement places no restrictions on $p_d$. Furthermore, by Lemma~\ref{lm: facet inclusion}, $F\backslash p_d$ is a facet of $\pm\partial B^{2k,i'}_{n'}$ for some $i'\leq k$ and $n'\leq n$. Thus it is enough to prove the statement for $d=2k$. We do this by induction.
	First we deal with the base cases: 
	\begin{itemize}
		\item If $k=i=1$ and $n\geq 2$, then $\partial B^{2,1}_n=\Delta^1_n$, and so both conditions hold by definition of $\Delta^1_n$.
		\item If $i=0$ and $k, n$ are arbitrary, then $\partial B^{2k, 0}_n=\partial \overline{\{-1, n-2k+1, \dots, n\}}$, so the statement holds.
		\item Finally, if $k$ is any number, $0\leq i\leq k$ and $n=2k$, then $|p_1|=1, |p_2|=2, \dots, |p_{2k}|=2k$; hence, the statement holds in this case as well.
	\end{itemize} 

	Now, for the inductive step, we assume that the statement holds for $k'=k-1$, all $i'\leq k'$ and $n'\geq 2k'$, and that it also holds for  $k'=k$, $i'=k$ and all $n'< n$. We will prove that then it holds for $k$, $n$, and all $i\leq k$. Note that the proof of Lemma \ref{lm: facet inclusion} implies the following: 
	\begin{enumerate}
	\item[($\star$)] if $j<d/2$ and $F\in \partial B^{d, j}_n$, then $F\backslash \{p_d\}\in \pm \partial B^{d-1, j}_{n'} \cup \pm \partial B^{d-1, j-1}_{n'}$, where $n'\in \{n-2, n-1\}$, and $|p_d|=n'+1\in\{n-1,n\}$. 
	\end{enumerate}
	Hence in the case of $d=2k$ and $i\leq k-1$, by applying the above statement to $F\in \partial B^{2k, i}_n$ twice, we obtain that $$F\backslash \{p_{2k-1}, p_{2k}\}\in \pm \Big(\partial B^{2k-2,i}_{n''} \cup \partial B^{2k-2, i-1}_{n''}\cup \partial B^{2k-2, i-2}_{n''}\Big), $$
where $|p_{2k-1}|=n''+1 \in \{n'-1,n'\}=\{|p_{2k}|-2, |p_{2k}|-1\}$. 
Thus the statement follows by the inductive hypothesis on $k$.

Finally we consider the case of $i=k$. 
Since $n>2k$, by definition of $\Delta^{2k-1}_n$,
\[\partial B^{2k, k}_n=\Delta^{2k-1}_n\subseteq \pm \Big(\partial B^{2k-1, k-1}_{n-1}*n \Big) \cup \Delta^{2k-1}_{n-1}.\]
Hence a facet $F\in \partial B^{2k,k}_n$ is either a facet of $\Delta^{2k-1}_{n-1}$ or it is a facet of $\pm(\partial B^{2k-1, k-1}_{n-1}*n)$. In the former case, the statement follows by the inductive hypothesis on $n$. In the latter case it follows from ($\star$) applied to $\partial B^{2k-1,k-1}_{n-1}$ and the inductive hypothesis on $k$. 
\end{proof}

We now turn to discussing sufficient conditions: the following lemma describes a large chunk of facets of $\Delta^{2k-1}_n$  and in fact characterizes all positive facets of $\Delta^{2k-1}_n$. For a simplicial complex $\Gamma$ on $V_n$, we denote by $\Gamma_+$ the restriction of $\Gamma$ to the positive vertices, i.e., to $[n]=\{1,2,\ldots,n\}$. A facet of $\Gamma$ is called \emph{positive} if $F\in \Gamma_+$.

\begin{lemma}\label{lm: positive facets}
	Let $k\geq 1$, $n\geq 2k$, and $1\leq m\leq k$ be integers. Let $S(2k,n)_m$ be the collection of subsets of $V_n$ of the form $\{p_1,p_2, p_3, \ldots, p_{2k}\}$ that satisfy the following conditions:
	\begin{enumerate}
		\item $|p_1|<|p_2| <\cdots <|p_{2k}|$;
		\item  $p_{2i-1}$ and $p_{2i}$ have the same sign for all $i=1,2,\ldots, k$;
		\item $|p_2|-|p_1|=1$,  
		$|p_{2i}|-|p_{2i-1}|=2$ for all $2\leq i \leq m$, and  $\{|p_{2i}|, |p_{2i-1}|\}=\{n-2(k-i)-1, n-2(k-i)\}$ for all $m< i \leq k$.
	\end{enumerate}
Let $S(2k,n)=\cup_{m=1}^k S(2k,n)_m$.
	Then any positive facet of $\Delta^{2k-1}_n$ is in $S(2k,n)$. Furthermore, any set in $S(2k,n)$ is a facet of $\Delta^{2k-1}_n$.
\end{lemma}

\noindent For instance, if $n\geq 9$, then $F=\{-3,-4, 5,7, n-1,n\}$ and $G=\{1,2,3,5,7,9\}$  are facets
of $\Delta^{5}_n$  since $F\in S(6,n)_2$ and $G\in S(6,n)_3$. Similarly, if $n\geq 6$, then $H=\{1,2, -(n-3), -(n-2), n-1, n\}$ is a facet of $\Delta^{5}_n$ since $H\in S(6,n)_1$.

\begin{proof} 
We start with the first statement. We use induction on $k$. In the base case of $k=1$, the statement follows from the definition of $\Delta^1_n$. In fact, the elements of $S(2,n)_1$ are precisely the facets of $\Delta^1_n\backslash \pm B^{1,0}_n$. For $k\geq 2$, a facet $F\in(\Delta^{2k-1}_n)_+$ can only be of the following two types:

\smallskip\noindent{\bf Case 1:}
	$F\in \left(\Delta^{2k-1}_n\backslash\pm B^{2k-1, k-1}_n\right)_+$. We will prove by induction on $k$ that in this case, $F$ is in $S(2k,n)_k$, and, in particular, that $\{p_{2k-1}, p_{2k}\}=\{j-2, j\}$  for some $j\leq n$. Since the only positive facet in $\Delta^{2k-1}_{2k}$ is $[2k]$ and $[2k]\in B^{2k-1, k-1}_{2k}$, it follows from equation \eqref{eq:decomposition} that $F\in \left(\pm(\partial B^{2k-1, k-1}_{j-1} * j) \backslash \pm B^{2k-1, k-1}_j\right)_+$ for some $2k<j\leq n$. 
		
Note that 
\begin{eqnarray*}
&&\left(\pm \big(\partial B^{2k-1, k-1}_{j-1} * j\big)\right)_+ 
=\left(\partial B^{2k-1, k-1}_{j-1}\right)_+ * j \\
	&\qquad& =\left(\Big(\partial B^{2k-2, k-1}_{j-2}\Big)_+ *(j-1,j)\right) \cup 
	\left(\Big(B^{2k-2, k-1}_{j-2}\backslash -B^{2k-2, k-2}_{j-2}\Big)_+*j \right),
\end{eqnarray*}		
where the last step is by Lemma~\ref{lm: prop of B^{d, i}_n}(3). 
On the other hand, by equation \eqref{B^d-in-terms of-B^{d-2}} and by Definition~\ref{def},
\begin{equation*}
\left(\pm B^{2k-1, k-1}_j\right)_+= \Big(B^{2k-3,k-1}_{j-2}\cup B^{2k-3,k-2}_{j-2}\Big)_+ * (j-1,j)=
\Big(\Delta^{2k-3}_{j-2}\Big)_+ * (j-1,j).
\end{equation*}
Comparing the last two equations and using the fact that $\partial B^{2k-2, k-1}_{j-2}=\Delta^{2k-3}_{j-2}$, we conclude that 
\begin{eqnarray*}
F &\in & \Big(B^{2k-2, k-1}_{j-2}\backslash -B^{2k-2, k-2}_{j-2}\Big)_+*j \\
&=&\Big(B^{2k-3,k-1}_{j-3}\backslash B^{2k-3,k-3}_{j-3}\Big)_+ * (j-2,j)\\
&=&\Big(\big(\Delta^{2k-3}_{j-3}\backslash \pm B^{2k-3,k-2}_{j-3}\big)_+ \cup 
\big(-B^{2k-3,k-2}_{j-3}\backslash B^{2k-3,k-3}_{j-3}\big)_+\Big)  * (j-2,j)\\
&=&\Big(\Delta^{2k-3}_{j-3}\backslash \pm B^{2k-3,k-2}_{j-3}\Big)_+ * (j-2,j).
\end{eqnarray*}
The last step uses that $\big(-B^{2k-3,k-2}_{j-3}\big)_+=\big(B^{2k-4,k-3}_{j-4}\big)_+*(j-3)=\big(B^{2k-3,k-3}_{j-3}\big)_+$.
This computation along with the inductive assumption shows that $F\backslash\{j-2,j\}\in S(2k-2,j-3)_{k-1}$, and hence that $F\in S(2k,j)_k\subseteq S(2k,n)$.

\smallskip\noindent{\bf Case 2:} $F\in \left(\pm B^{2k-1, k-1}_n\right)_+$. Then 
		\begin{eqnarray*}
		F &\in& \left(B^{2k-2, k-1}_{n-1} \cup B^{2k-2, k-2}_{n-1}\right)_+*n \\
		&=& \left(B^{2k-3, k-1}_{n-2}\cup B^{2k-3, k-2}_{n-2}\right)_+*(n-1, n)= \left(\Delta^{2k-3}_{n-2}\right)_+*(n-1, n),
		\end{eqnarray*}
		and the assertion again follows by the inductive hypothesis. This concludes the proof of the first statement.

	For the second statement, we also use induction on $k$. We start by showing that all elements of $S(2k,n)_k$ are facets of $\Delta^{2k-1}_n \backslash \pm B^{2k-1, k-1}_n$. This claim does hold for $k=1$. For $k\geq 2$, note that by Lemma \ref{lm: prop of B^{d, i}_n}(2) and Definition \ref{def},
	\[ \Delta^{2k-3}_{j-3}\backslash \pm B^{2k-3, k-2}_{j-3} \subseteq \Delta^{2k-3}_{j-3}\backslash 
	\big(B^{2k-3, k-2}_{j-3} \cup B^{2k-3, k-3}_{j-3}\big)=B^{2k-3, k-1}_{j-3}\backslash B^{2k-3, k-3}_{j-3} \mbox{ for all } 2k<j\leq n,\]
	and by symmetry, $\Delta^{2k-3}_{j-3}\backslash \pm B^{2k-3,k-2}_{j-3}$ is also contained in $-\big(B^{2k-3,k-1}_{j-3}\backslash B^{2k-3,k-3}_{j-3}\big)$.
	This together with Definition \ref{def}, Lemma \ref{lm: prop of B^{d, i}_n}(3), and equation \eqref{eq:decomposition} implies that for $2k<j\leq n$,
	\begin{equation*}
	\begin{split}
	\left(\Delta^{2k-3}_{j-3}\backslash \pm B^{2k-3, k-2}_{j-3}\right)*\Big((j-2,j)\cup (-j+2, -j)\Big)\subseteq \pm  \left( \big(B^{2k-2, k-1}_{j-2}\backslash -B^{2k-2, k-2}_{j-2}\big)*j \right) \\
	\subseteq  \pm \left(\partial B^{2k-1, k-1}_{j-1} * j\right) \backslash \pm B^{2k-1, k-1}_j \subseteq \Delta^{2k-1}_n\backslash \pm B^{2k-1, k-1}_n.
	\end{split}
	\end{equation*}
	Since $S(2k,n)_k=\bigcup_{j=2k+1}^n \left\{F\cup\{j-2,j\}, \,  F\cup\{-j+2,-j\} \, : \, F\in S(2k-2,j-3)_{k-1}\right\}$, the claim follows by the inductive hypothesis. 
	
	Finally, by Lemma \ref{lm: prop of Delta^{2k-1}_n},
	$$\Delta^{2k-3}_{n-2}*\Big((n-1,n)\cup (-n+1, -n)\Big)\subseteq \Delta^{2k-1}_n.$$
Since $\bigcup_{m=1}^{k-1} S(2k,n)_m = \{F\cup\{n-1,n\}, \  F\cup\{-n+1,-n\} \ : \ F\in S(2k-2,n-2)\}$, the above equation together with the inductive hypothesis shows that all elements of $\bigcup_{m=1}^{k-1} S(2k,n)_m$ are also facets of $\Delta^{2k-1}_n$. The result follows.
\end{proof}

We close this section with a couple of remarks related to Lemma \ref{lm: positive facets}.

\begin{remark}  \label{rem:typical}
	For a sufficiently large $n$, the set $S(2k,n)_k$ of Lemma \ref{lm: positive facets} describes a majority of facets of $\Delta^{2k-1}_n$. Indeed, by Lemma \ref{lm: facet count}, there are $2^k\binom{n}{k}+O(n^{k-1})$ facets in any combinatorial $(2k-1)$-sphere on $V_n$ that is cs-$k$-neighborly. On the other hand, the cardinality of the set $S(2k,n)_k$ is $2^k\binom{n-1-2(k-1)}{k}=2^k\binom{n}{k}+O(n^{k-1})$; here $2^k$ is the number of ways to attach signs to the $k$ pairs $(|p_{2j-1}|, |p_{2j}|)$, and $\binom{n-1-2(k-1)}{k}$ is the number of ways to choose $|p_1|, |p_3|,\dots, |p_{2k-1}|$ from $[n]$.
\end{remark}

\begin{remark}
	For any $k\geq 4$ and $n>2k$, the complex $(\Delta^{2k-1}_n)_+$ is not pure. Indeed, since $\Delta^{2k-1}_n$ is cs-$k$-neighborly, the set $\tau=\{1,2,\dots,k\}$ is a face of $(\Delta^{2k-1}_n)_+$. However, according to Lemma~\ref{lm: positive facets}, no $(2k-1)$-dimensional face of $(\Delta^{2k-1}_n)_+$ contains $\tau$. Furthermore, for $k= 3$ and $n\geq 9$, $(\Delta^{5}_n)_+$
	is not a (combinatorial) ball. To see this, note that the intersection of the facet $\{1,2,3,5,6,8\}$ with the complex generated by other positive facets of $\Delta^{5}_n$ is not a pure $4$-dimensional complex because, as follows from Lemma \ref{lm: positive facets}, the $3$-face $\{1,2,6,8\}$ is a maximal face of this intersection.	
\end{remark}

\section{An overview of two high-dimensional constructions} \label{sec:two-constructions}
This is an intermission section: with Lemmas \ref{lm: facet inclusion}--\ref{lm: positive facets} at our disposal, we are ready to outline both of our high-dimensional constructions while deferring most of the proofs to the following two sections. The first construction is very simple:

\begin{definition} \label{def:Lambda}
For $d\geq 1$ and $n\geq d+1$, define $\Lambda^{d}_n:=\lk\big(\{1,2\}, \Delta^{d+2}_{n+2}\big)$.
\end{definition}

It follows from the definition that the complex $\Lambda^{d}_n$ is a combinatorial $d$-sphere whose vertex set is contained in $W_n:=\{\pm 3,\pm 4,\ldots,\pm (n+2)\}$. We will mostly concentrate on the $d=2k-1$ case. The significance of $\Lambda^{2k-1}_n$ is that it gives us a new construction of a cs combinatorial $(2k-1)$-sphere with $2n$ vertices that is cs-$k$-neighborly:

\begin{theorem} \label{thm:Lambda}
Let $k\geq 1$ and $n\geq 2k$ be integers. Then $\Lambda^{2k-1}_n$ is a combinatorial $(2k-1)$-sphere whose vertex set is $W_n$. This sphere is both cs and cs-$k$-neighborly (w.r.t.~$W_n$). Furthermore, if $k\geq 2$ and $n$ is sufficiently large, then $\Lambda^{2k-1}_n$ is not isomorphic to $\Delta^{2k-1}_n$.
\end{theorem}

\noindent The fact that there exists an edge of $\Delta^{2k+1}_{n+2}$ whose link is {\em cs} is already rather surprising; even more surprising is the fact that this link is {\em cs-$k$-neighborly}.

The complete proof of Theorem \ref{thm:Lambda} will be given in the next section. Here we merely explain 
why $\Lambda^{2k-1}_n$ could be cs and cs-$k$-neighborly. According to Remark \ref{rem:typical}, we expect that most facets $G$ of $\Lambda^{2k-1}_n$ satisfy $\{1,2\}\cup G \in S(2k+2,n+2)$. In such a case, by definition of $S(2k+2,n+2)$, the set $\{1,2\}\cup (-G)$ is also in $S(2k+2,n+2)$, which by Lemma \ref{lm: positive facets} implies that $-G\in \Lambda^{2k-1}_n$. Similarly, to see that a ``generic" $k$-element subset of $W_n$ is a face of $\Lambda^{2k-1}_n$, let $\tau=\{i_1,i_2,\dots i_k\}\subseteq W_n$ be such that $|i_k|\leq n$ and $|i_{s+1}|-|i_s|\geq 3$ for all $1\leq s\leq k-1$. For $1\leq s\leq k$, define $j_s$ to be the integer that has the same sign as $i_s$ and satisfies $|j_s|-|i_s|=2$, and let $G:=\{i_1,j_1,\dots,i_k,j_k\}$. Then the set $\{1,2\}\cup G$ is in $S(2k+2,n+2)$,  hence $G$ is a face of $\Lambda^{2k-1}_n$, and hence so is $\tau\subseteq G$. 

The proof that $\Lambda^{2k-1}_n$ and $\Delta^{2k-1}_n$ (for $n\gg 0$) are non-isomorphic will also be based on results of Section 3.2: in the spirit of the previous paragraph, we will use Lemma \ref{lm: positive facets} to show that quite a few $(2k-3)$-faces of $\Lambda^{2k-1}_n$, namely $\Omega(n^{k-1})$ of them, have large links: such links are (graph-theoretic) cycles, and large here means that the link contains $2(n-3k+3)$ or more vertices. On the other hand, we will show using Lemmas \ref{lm: facet} and \ref{lm: facet inclusion} that only $O(n^{k-2})$ of $(2k-3)$-faces of $\Delta^{2k-1}_n$ have such large links.

We close our discussion of $\Lambda^{2k-1}_n$ with one additional property of $\Lambda^{2k-1}_n$. (In contrast, we suspect that for a sufficiently large $n$, $\Delta^{2k-1}_n$ doesn't have this property.) Let $F(2k,n)$ be the collection of facets of the cyclic polytope $C(2k,n)$ of the form $\{i_1, i_1+1\}\cup \{i_2, i_2+1\}\cup \dots \cup \{i_k, i_k+1\}$, where $i_1 \geq 1$, $i_k<n$ and $i_{m+1}\geq i_m+2$ for all relevant $m$. Following Kalai \cite{Kal}, we denote by $\B(F(2k,n))$ the simplicial complex generated by the facets in $F(2k,n)$. This complex is a combinatorial $(2k-1)$-ball: it belongs to the family of \emph{squeezed balls} defined in \cite{Kal}; in fact, it is the inclusion largest ball among all squeezed balls of dimension $2k-1$ with at most $n$ vertices.  The squeezed balls were instrumental in Kalai's proof that for a large $n$ and $d\geq 5$, most of combinatorial $(d-1)$-spheres on $n$ vertices are not polytopal, that is, they cannot be realized as the boundary complexes of polytopes.

\begin{proposition}  \label{prop:squeezed-balls-as-subcomplexes}
	Let $k\geq 1$ and $n\geq k+1$. 
	The sphere $\Lambda^{2k-1}_{2n-1}$ contains an isomorphic image of the squeezed ball $\B(F(2k,n))$ as a subcomplex. Consequently, all squeezed $(2k-1)$-balls with $\leq n$ vertices are embeddable in $\Lambda^{2k-1}_{2n-1}$, and hence also in $\Delta^{2k+1}_{2n+1}$, as subcomplexes.
\end{proposition}
\noindent For a set $A$ and an integer $s$, we denote by $\binom{A}{s}$ the collection of all $s$-element subsets of $A$.

\begin{proof} 
Let $\rho: [n] \to W_{2n-1}=\{\pm 3, \ldots, \pm(2n+1)\}$ be defined by $i\mapsto 2i+1$. Then $\rho$ is an injection, and so is the induced map from $\binom{[n]}{2k}$ to $\binom{W_{2n-1}}{2k}$, which we also denote by $\rho$. Furthermore, it follows from the definitions of $F(2k, n)$ and $S(2k+2,2n+1)$ that for any $G\in F(2k,n)$, the set $\{1,2\}\cup \rho(G)$ is in $S(2k+2,2n+1)$, and hence   $\rho(G)\in \Lambda^{2k-1}_{2n-1}$. Thus $\rho$ embeds $\B(F(2k,n))$ into $\Lambda^{2k-1}_{2n-1}$.
\end{proof}

Next we outline how to construct many cs combinatorial $(2k-1)$-spheres that are cs-$(k-1)$-neighborly. This involves applying bistellar flips to $\Delta^{2k-1}_n$ and requires the following lemma.

\begin{lemma} \label{lem:F_i-G_i-flip}
	Let $k\geq 2$ and $3\leq i\leq n-4k+3$. Let $F_i=\{i, i+3, i+7, i+11, \dots, i+4k-5\}$ and $G_i=\{i-1, i+1, i+5, i+9, \dots, i+4k-3\}$. Then $F_i$ is a $(k-1)$-face of $\Delta^{2k-1}_n$, $G_i$ is not a face of $\Delta^{2k-1}_n$, and $\lk\big(F_i, \Delta^{2k-1}_n\big)=\partial {\overline{G_i}}$. 
\end{lemma}
\begin{proof}
	The set $F_i$ is a face of $\Delta^{2k-1}_n$ since $\Delta^{2k-1}_n$ is cs-$k$-neighborly. Lemma \ref{lm: positive facets} shows that the link of $F_i$ contains the $(k-1)$-sphere $\partial{\overline{G_i}}$, and hence the link is $\partial{\overline{G_i}}$. Assume $G_i$ is a face of $\Delta^{2k-1}_n$, and let $G=\{p_1,\ldots,p_{2k}\}$ be any facet of $\Delta^{2k-1}_n$ that contains $G_i$, where $|p_1|<|p_2|<\cdots<|p_{2k}|$. Since $i-1>1$ and $(i+1)-(i-1)> 1$, it follows from Lemma \ref{lm: facet} that either $\{|p_1|,|p_2|\}=\{i-1,i\}$ or $|p_2|\leq i-1$. In either case, $i+1\geq |p_3|$. Furthermore, since every two consecutive elements of $G_i\backslash \{i-1\}$ are apart from each other by four,  Lemma \ref{lm: facet} also implies that $|G_i \cap \{p_{2j-1}, p_{2j}\}| \leq 1$ for any $2\leq j\leq k$. But this is impossible since $|G_i \backslash \{i-1\}|=k$ and all of these $k$ elements must belong to the union $\cup_{j=2}^k \{p_{2j-1},p_{2j}\}$ of $k-1$ pairs. This proves that $G_i\notin \Delta^{2k-1}_n$.
\end{proof}

According to Lemma \ref{lem:F_i-G_i-flip}, replacing $\pm\st(F_i,\Delta^{2k-1}_n)= \pm\big(\overline{F_i} * \partial\overline{G_i}\big)$ with $\pm\big(\partial\overline{F_i}*\overline{G_i}\big)$ constitutes admissible bistellar flips, and, by symmetry, the resulting combinatorial sphere is cs. Furthermore, since the stars of distinct $F_i$'s share no common facets, we could simultaneously apply such pairs of symmetric flips for all $i$ in an arbitrary subset $J$ of $\{3,4,\ldots, n-4k+2\}$. We will see in Theorem~\ref{thm:cs-(k-1)-neighb-constr} that for $k\geq 3$, the spheres produced in this way are cs-$(k-1)$-neighborly and that they are pairwise non-isomorphic.

\section{The edge links of $\Delta^{2k+1}_n$} \label{sec:k-neighb}
The goal of this section is to prove Theorem \ref{thm:Lambda}. Along the way, we investigate a more general question of what edges of $\Delta^{2k+1}_n$ have highly cs-neighborly links. Recall that according to Lemma~\ref{lm: prop of Delta^{2k-1}_n}, the link of $\{n-1,n\}$ in $\Delta^{2k+1}_n$ is $\Delta^{2k-1}_{n-2}$, and hence it is both cs and cs-$k$-neighborly. Are there other edges of $\Delta^{2k+1}_n$ that possess the same properties?

In the case of $\Delta^3_n$, the answer is given by Corollary \ref{cor: edge links}. We will need the following extension.

\begin{lemma}\label{lm: cs 1 neighborly edge links}
	Let $n$ be sufficiently large and let $e$ be an edge of $\Delta^3_n$. The link $\lk(e, \Delta^{3}_n)$ has $2n-4$ vertices if and only if $e=\pm \{1,2\}$ or $\pm \{n-1,n\}$. Furthermore, both balls $\lk(e, B^{3,2}_n)$ and $\lk(e, -B^{3,2}_n)$ are cs-$1$-neighborly (w.r.t.~$V_n\backslash \pm e$) and $1$-stacked if and only if $e=\pm \{1,2\}$.
\end{lemma}
\begin{proof}
	The first statement follows from Corollary \ref{cor: edge links}. The second one follows from the fact that $\lk(\{1,2\}, \pm B^{3,1}_n)=\pm (n-1,n)$, while $\lk(\{n-1, n\}, B^{3,1}_n)= B^{1,1}_{n-2}$, see equation \eqref{eq:facets-of-B^{3,1}}; in particular, $\lk(\{n-1,n\}, B^{3,2}_n)$ is only cs-$0$-neighborly. 
\end{proof}

Our first task is to generalize this lemma and prove in Theorem \ref{lm: edge link main result} below that for all $k\geq 2$ and sufficiently large $n$, the only edges $e$ whose links in $ \Delta^{2k-1}_n$ are both cs and cs-$k$-neighborly are $e=\pm \{1,2\}$ and $\pm \{n-1,n\}$. 
 The proof relies on the following series of lemmas. We start by analyzing vertex links of $B^{d,i}_n$.

\begin{lemma}\label{lm: vertex link in B}
	Let $d\geq 2$, let $1\leq i\leq \lceil \frac{d}{2}\rceil$, and let $n$ be sufficiently large. Then for $\ell\in V_n$, the link $\lk(\ell, B^{d, i}_n)$ is cs-$i$-neighborly (w.r.t.~$V_n\backslash \{\pm \ell\}$) if and only if $\ell=n$ and $i\leq \lfloor d/2\rfloor$.
\end{lemma}
\begin{proof}
	The proof is by induction on $d$ and $i$. First we deal with the base case of $d=2k-1$ and $i=k$. The link $\lk(\ell, B^{2k-1,k}_n)$ is a combinatorial $(2k-2)$-ball and hence it has only $O(n^{k-1})$ faces of dimension $k-1$. In particular, for a sufficiently large $n$, it cannot be cs-$k$-neighborly. (To see that a $(2k-2)$-ball $\Gamma$ on $m$ vertices has only $O(m^{k-1})$ faces of dimension $k-1$, consider a new vertex $v$ and let $\hat{\Gamma}:=\Gamma \cup (v*\partial \Gamma)$. The resulting complex is a $(2k-2)$-sphere on $m+1$ vertices, and so by the Upper Bound Theorem for spheres, see \cite{Stanley75}, $f_{k-1}(\Gamma)\leq f_{k-1}(\hat{\Gamma})=O(m^{k-1})$.)
	
	For $i\leq d/2$, it follows from the definition of $B^{d, i}_n$ that
	\begin{equation}\label{eq: 5.1}
		\lk(\ell, B^{d, i}_n)=
		\begin{cases}
		B^{d-1, i}_{n-1} & \text{if $\ell=n$,}\\
		-B^{d-1, i-1}_{n-1} & \text{if $\ell=-n$,}\\
		\left(\lk(\ell, B^{d-1, i}_{n-1})*n\right)\cup \left(\lk(\ell, -B^{d-1, i-1}_{n-1})*(-n)\right) & \text{if $|\ell|<n$}.
		\end{cases}
	\end{equation}
    Hence in the second base case of $d=2$ and $i=1$, the link $\lk(n, B^{2,1}_n)$ is indeed cs-$1$-neighborly w.r.t.~$V_{n-1}$ while for $\ell \neq n$, the link $\lk(\ell, B^{2,1}_n)$ has at most five vertices and hence for a sufficiently large $n$, it is not cs-$1$-neighborly. 
    
    Now assume that the statement holds for $d'<d$ and $i'\leq \lceil \frac{d'}{2}\rceil$. 
We will prove that then it also holds for $d,i$. By Lemmas \ref{lm: prop of B^{d, i}_n}(1) and \ref{remark on B^{d, i}_n}, and by (\ref{eq: 5.1}), the link of $n$ is indeed cs-$i$-neighborly, while the link of $-n$ is cs-$(i-1)$-neighborly but not cs-$i$-neighborly. So assume that $|\ell|\leq n-1$. By Lemma \ref{lm: prop of B^{d, i}_n}(2), $\lk(\ell, -B^{d-1, i-1}_{n-1})$ is a subcomplex of $\lk(\ell, B^{d-1, i}_{n-1})$. Thus, for $\lk(\ell, B^{d, i}_n)$ to be cs-$i$-neighborly, $\lk(\ell, B^{d-1, i}_{n-1})$ must be cs-$i$-neighborly. Hence, by the inductive hypothesis, $\ell$ must be $n-1$. However, for $i=1$, the link $\lk(n-1, -B^{d-1, 0}_{n-1})=\emptyset$ while for $i>1$, it follows from the inductive hypothesis that $\lk(n-1, -B^{d-1, i-1}_{n-1})$ is not cs-$(i-1)$-neighborly. This implies that $\lk(\ell, B^{d, i}_{n})$ is not cs-$i$-neighborly if $|\ell|\leq n-1$.
\end{proof}

With Lemma \ref{lm: vertex link in B} in hand, we are ready to investigate the edge links of $B^{d,i}_n$. We start by proving the first part of Theorem \ref{thm:Lambda}.

\begin{lemma}\label{lm: link of 12}
	Let $d\geq 2$, $1\leq i\leq \lceil \frac{d}{2}\rceil$ and $n\geq d+1$.
	\begin{enumerate}
		\item The complexes $\lk(\{1,2\}, B^{d, i}_n)$ and $\lk(\{1,2\}, -B^{d,i}_n)$ are antipodal complexes that are cs-$(i-1)$-neighborly and $(i-1)$-stacked. Furthermore, if $i\leq  d/2 $, then they share no common facets.
		\item The complex $\Lambda^d_{n}$ (i.e., the link of $\{1,2\}$ in $\Delta^{d+2}_{n+2}$) is both cs and cs-$\lceil\frac{d}{2}\rceil$-neighborly.
	\end{enumerate}
\end{lemma}
\begin{proof}
	The proof of the first statement is by induction on $d$ and $i$. We begin with the base case of $i=1$ and $d\geq 2$. It follows from the definition of $B^{d,1}_n$ and the fact that $\lk(\{1,2\}, B^{d, 0}_n)=\emptyset$ that 
	$$\lk(\{1,2\}, B^{d, 1}_n)=\lk(\{1,2\}, B^{d-1, 1}_{n-1})*n=\dots= \overline{\{n-d+2, n-d+3, \dots, n\}}.$$ Similarly, $\lk(\{1,2\}, -B^{d, 1}_n)=-\overline{\{n-d+2, n-d+3, \dots, n\}}$. Hence the link of $\{1,2\}$ in $B^{d, 1}_n$ and the link of $\{1,2\}$ in $-B^{d, 1}_n$ form antipodal complexes that are cs-$0$-neighborly and $0$-stacked. The other base case $d=3$ and $i=2$ follows from Lemma \ref{lm: cs 1 neighborly edge links}.
	
	For $2\leq i\leq  d/2$, by definition of $B^{d, i}_n$,  $\lk(\{1,2\}, B^{d, i}_n) =D_1\cup D_2$, where $$D_1=\lk(\{1,2\}, B^{d-1, i}_{n-1})*n \;\;\text{and}\;\; D_2=\lk(\{1,2\}, -B^{d-1, i-1}_{n-1})*(-n). $$
	By the inductive hypothesis, $D_1$ is cs-$(i-1)$-neighborly and $(i-1)$-stacked while $D_2$ is cs-$(i-2)$-neighborly and $(i-2)$-stacked. Furthermore, by Lemma \ref{lm: prop of B^{d, i}_n}(2), $-B^{d-1, i-1}_{n-1}\subseteq B^{d-1, i}_{n-1}$ and hence $D_1\cap D_2=\lk(\{1,2\}, -B^{d-1, i-1}_{n-1})$ is $(i-2)$-stacked. It follows from Lemma \ref{lm: stackedness} that $\lk(\{1,2\}, B^{d, i}_n)$ is indeed cs-$(i-1)$-neighborly and $(i-1)$-stacked. Also by the inductive hypothesis, $\lk(\{1,2\}, B^{d-1, j}_{n-1})$ and $\lk(\{1,2\}, -B^{d-1, j}_{n-1})$ are antipodal complexes for $j=i-1$ or $i$. Therefore $\lk(\{1,2\}, -B^{d,i}_n)=(-D_1)\cup (-D_2)=-\lk(\{1,2\}, B^{d,i}_n)$. Finally, for $i\leq  \frac{d}{2}$, the complexes $B^{d-1, i}_{n-1}$ and $B^{d-1, i-1}_{n-1}$ share no common facets. Hence $D_1$ and $-D_2$, as well as $-D_1$ and $D_2$, also share no common facets; thus, neither do $\lk(\{1,2\}, B^{d,i}_n)$ and $\lk(\{1,2\}, -B^{d,i}_n)$. We will treat the case $d=2k-1$ and $i=k$ a bit later.
	
	Next we prove that $\Lambda^{2k-1}_{n}$ is both cs and cs-$(k-1)$-neighborly by induction on $n$. First $\Lambda^{2k-1}_{2k}$ is the boundary of a $2k$-dimensional cross-polytope, and so it is both cs and cs-$k$-neighborly. In the inductive step, to obtain $\Delta^{2k+1}_{n+3}$ from $\Delta^{2k+1}_{n+2}$, we delete $\pm B^{2k+1, k}_{n+2}$ and insert $\pm \big(\partial B^{2k+1, k}_{n+2} *(n+3)\big)$. On the level of edge links, by the inductive hypothesis we start with the $(2k-1)$-sphere $\Lambda^{2k-1}_n$ that is both cs and cs-$k$-neighborly. We then delete the cs-$(k-1)$-neighborly and $(k-1)$-stacked balls $\lk\big(\{1,2\}, \pm B^{2k+1, k}_{n+2}\big)$ that are antipodal and share no common facets, and insert the cones over the boundary of these two balls. Thus, the resulting complex is also cs; furthermore, by Lemma \ref{lm: induction method}, it is cs-$k$-neighborly. In the case of $d=2k$, note that by Proposition \ref{prop: odd even sphere relation}, $\Delta^{2k+1}_{n+2}\subseteq \Delta^{2k+2}_{n+2}$. Hence $\Lambda^{2k}_n\supseteq \Lambda^{2k-1}_n$, and so $\Lambda^{2k}_n$ is also cs-$k$-neighborly. The proof that $\Lambda^{2k}_n$ is cs is identical to the proof in the odd-dimensional cases.
	
Finally, to complete the proof of the first part for the case of $d=2k-1$ and $i=k$, note that,
	$$\lk\big(\{1,2\}, B^{2k-1, k}_n\big)=\lk\big(\{1,2\}, \Delta^{2k-1}_n\big)\backslash \lk\big(\{1,2\}, B^{2k-1, k-1}_n\big).$$ 
We then conclude from the case of $d=2k-1, i=k-1$ and Lemma \ref{lm: complement} that $\lk\big(\{1,2\}, \pm B^{2k-1, k}_n\big)$ is indeed cs-$(k-1)$-neighborly and $(k-1)$-stacked.
\end{proof}

\begin{lemma}\label{lm: neighborly stacked edge links}
		Let $k\geq 3$ and let $n$ be sufficiently large. The only edges $e\subseteq V_{n}$ such that both $\lk(e, B^{2k-1, k-1}_n)$ and $\lk(e, -B^{2k-1, k-1}_n)$ are cs-$(k-2)$-neighborly and $(k-2)$-stacked are $e=\pm \{1,2\}$.
\end{lemma}
\begin{proof}
	We prove the statement by considering the following three cases.
	
	\smallskip\noindent{\bf Case 1:} $e\subseteq V_{n-2}$. By equation (\ref{B^d-in-terms of-B^{d-2}}),
	\begin{equation*}
	\begin{split}
	\lk(e, B^{2k-1, k-1}_n)&=\left( \lk(e, B^{2k-3,k-1}_{n-2})*(n-1, n)\right) \cup \left( \lk(e, -B^{2k-3, k-2}_{n-2})*(n, -n+1, -n)\right) \\
	&\cup \left( \lk(e, B^{2k-3, k-3}_{n-2})*(-n, n-1)\right) .
	\end{split}
	\end{equation*}
    By Lemma \ref{lm: prop of B^{d, i}_n}(2), $\lk(e, B^{2k-3,k-1}_{n-2}) \supseteq \lk(e, -B^{2k-3,k-2}_{n-2})\supseteq \lk(e, B^{2k-3,k-3}_{n-2})$. Hence according to Lemmas \ref{lm: stackedness} and \ref{lm: exact stackedness}, for the link $\lk(e, B^{2k-1, k-1}_n)$ to be cs-$(k-2)$-neighborly and $(k-2)$-stacked, we must have 
	\begin{itemize}
		\item $\lk(e, B^{2k-3, k-1}_{n-2})$ is cs-$(k-2)$-neighborly and $(k-2)$-stacked;
		\item $\lk(e, -B^{2k-3, k-2}_{n-2})$ is cs-$(k-3)$-neighborly and $(k-3)$-stacked.
	\end{itemize}
	Note that the above two bullet points also apply to $\lk(e, -B^{2k-3,k-1}_{n-2})$ and $\lk(e, B^{2k-3,k-2}_{n-2})$, resp.  This leads to a) $\lk(e, \Delta^{2k-3}_{n-2})$ is cs-$(k-2)$-neighborly, and b) both links $\lk(e,  B^{2k-3, k-2}_{n-2})$ and $\lk(e,  -B^{2k-3, k-2}_{n-2})$ are cs-$(k-3)$-neighborly and $(k-3)$-stacked. By Lemma \ref{lm: cs 1 neighborly edge links} and its proof, the only edges $e$ of $B^{3,1}_{n-2}$ that satisfy both a) and b) are $\pm \{1,2\}$. This proves the base case. On the other hand, the inductive hypothesis implies that the only edges that satisfy condition b) are $\pm\{1,2\}$.  Finally by Lemma \ref{lm: link of 12}, the links of $e=\pm \{1,2\}$ do have the desired properties.
	
	\smallskip\noindent{\bf Case 2:} $e=\pm \{i, n-1\}$, where $|i|<n-1$. For $e=\{i, n-1\}$, 
	\begin{equation*}
	\begin{split}
	\lk(e, -B^{2k-1, k-1}_n)&=\left( \lk(e, -B^{2k-2, k-1}_{n-1})*(-n)\right)  \cup \left( \lk(e, B^{2k-2, k-2}_{n-1})*n\right) \\
	&=\left( \lk(i, B^{2k-3, k-2}_{n-2})*(-n)\right)  \cup \left( \lk(i, B^{2k-3, k-2}_{n-2})*n\right) .
	\end{split}
	\end{equation*}
	Thus, by Lemma \ref{lm: vertex link in B}, the link of $e$ is cs-$(k-2)$-neighborly only if $i=n-2$. But then the link equals the suspension of $B^{2k-4, k-2}_{n-3}$. Since by Lemma \ref{remark on B^{d, i}_n}, $B^{2k-4, k-2}_{n-3}$ is exactly $(k-2)$-stacked, we conclude that in this case the link is exactly $(k-1)$-stacked. 
	
\smallskip\noindent	{\bf Case 3:} $e=\pm \{i, n\}$, where $|i|<n$. For $e=\{i, n\}$, $\lk(e, -B^{2k-1, k-1}_n)=\lk(i, B^{2k-2, k-2}_{n-1})$ and by Lemma \ref{lm: vertex link in B}, it is cs-$(k-2)$-neighborly only if $i=n-1$. However, in this case $\lk(e, B^{2k-1, k-1}_n)=B^{2k-3, k-1}_{n-2}$, and so according to Lemma \ref{remark on B^{d, i}_n} it is not $(k-2)$-stacked. 
\end{proof}

We are now in {\color{blue}a} position to prove the promised result.
\begin{theorem}\label{lm: edge link main result}
	Let $k\geq 2$ and let $n$ be sufficiently large. The only edges $e\subseteq V_n$ whose links in $\Delta^{2k-1}_n$ are cs-$(k-1)$-neighborly are $e=\pm \{1,2\}$ and $\pm \{n-1,n\}$. Furthermore, the links of these four edges are also cs.
\end{theorem}

\begin{proof}
The links of $\pm\{1,2\}$ are cs by Lemma \ref{lm: link of 12}; the links of $\pm\{n-1,n\}$ are cs because both of them are $\Delta^{2k-3}_{n-2}$. In the rest of the proof we concentrate on the first statement.

The case $k=2$ follows from Corollary \ref{cor: edge links}, so assume that $k\geq 3$. If $e=\pm \{i, n\}$, then $\lk(\{i, n\}, \Delta^{2k-1}_n)=\lk(i, \partial B^{2k-1, k-1}_{n-1}).$
Since $B^{2k-1, k-1}_{n-1}$ is $(k-1)$-stacked, it follows that in this case $\skel
_{k-2}\big(\lk(i, \partial B^{2k-1, k-1}_{n-1})\big)=\skel_{k-2}\big(\lk(i, B^{2k-1, k-1}_{n-1})\big)$. Hence by Lemma \ref{lm: vertex link in B}, the link is cs-$(k-1)$-neighborly if and only if $i=n-1$. Similarly, if $e=\{-i, -n\}$, then $\lk(e, \Delta^{2k-1}_n)$ is cs-$(k-1)$-neighborly if and only if $i=n-1$.
	
Now assume that $\pm n \notin e$. Since $\lk(e, \Delta^{2k-1}_n)$ is obtained from $\lk(e, \Delta^{2k-1}_{n-1})$ by replacing $\lk(e, B^{2k-1, k-1}_{n-1})$ with $\lk(e, \partial B^{2k-1, k-1}_{n-1}) *n$ and $\lk(e, -B^{2k-1, k-1}_{n-1})$ with $\lk(e, -\partial B^{2k-1, k-1}_{n-1}) *(-n)$, the link $\lk(e, \Delta^{2k-1}_n)$ is cs-$(k-1)$-neighborly if and only if $\lk(e, \Delta^{2k-1}_{n-1})$ is cs-$(k-1)$-neighborly and furthermore $\lk(e, \pm B^{2k-1, k-1}_{n-1})$ is cs-$(k-2)$-neighborly and $(k-2)$-stacked. By Lemma \ref{lm: neighborly stacked edge links}, the second condition is equivalent to $e=\pm \{1,2\}$. Finally, Lemma \ref{lm: link of 12} implies that $\lk(\{1,2\}, \Delta^{2k-1}_n)$ is indeed cs-$(k-1)$-neighborly.
\end{proof}

\begin{remark}
	While by Lemma \ref{lm: link of 12}, the link $\lk(\{1,2\}, \Delta^d_n)$ is always cs, the link $\lk(\{n-1,n\}, \Delta^{d}_n)$ is cs only when $d$ is odd. Indeed, $\lk(\{n-1,n\}, \Delta^{2k}_n)=\lk(n-1, \partial B^{2k, k-1}_{n-1})=\partial B^{2k-1,k-1}_{n-2}$ is not cs. 
\end{remark}

To complete the proof of Theorem \ref{thm:Lambda}, it is left to prove the following result.
\begin{theorem}\label{thm: distinct edge links}
	Let $k\geq 2$. For a sufficiently large $n$, the complexes $\Lambda^{2k-1}_{n-2}$ and $\Delta^{2k-1}_{n-2}$ are not isomorphic. 
\end{theorem}

Before proceeding, it is worth remarking that although we do not have a proof, we suspect that for $k\geq 2$ and $n\gg 0$, $\Lambda^{2k}_{n-2}$ and $\Delta^{2k}_{n-2}$ are also not isomorphic. It is also worth noting that for all $n\leq 6$, the spheres $\Lambda^3_n$ and $\Delta^3_n$ {\em are isomorphic}, while $\Lambda^3_7$ and $\Delta^3_7$ are already {\em not isomorphic}. This is not hard to check using, for instance, the computer program Sage.

Theorem \ref{thm: distinct edge links} is an immediate consequence of the following two lemmas which are independently interesting. We spend the rest of this section establishing these lemmas.
\begin{lemma}  \label{lemma:Omega(n^{k-1})}
		Let $k\geq 2$. For a sufficiently large $n$, the complex $\Lambda^{2k-1}_{n-2}$ has at least $\Omega(n^{k-1})$ faces of dimension $(2k-3)$ whose links have $2(n-3k+1)$ or more vertices.
\end{lemma}

\begin{lemma} \label{lemma:O(n^{k-2})}
	Let $k\geq 2$. For a sufficiently large $n$, the complex $\Delta^{2k-1}_{n-2}$ has at most $O(n^{k-2})$ faces of dimension $(2k-3)$ whose links have $2(n-3k+1)$ or more vertices.
\end{lemma}

\smallskip\noindent {\it Proof of Lemma \ref{lemma:Omega(n^{k-1})}: \ }
	Let $\sigma=\{p_1, p_2, \dots, p_{2k-3}, p_{2k-2}\}\subseteq V_n$ be such that a) $|p_1|\geq 7$ and $|p_{2k-2}|\leq n-4$, b) for each $i$, the elements $p_{2i-1}, p_{2i}$ have the same signs and $|p_{2i}|-|p_{2i-1}|=2$, and c) $|p_{2j+1}|-|p_{2j}|\geq 5$ for $1\leq j \leq k-2$. By Lemma \ref{lm: positive facets}, both $\sigma\cup \{1,2\}\cup \{m, m+2\}$ and $\sigma\cup \{1,2\}\cup \{-m, -m-2\}$ ($m> 0$) are facets of $\Delta^{2k+1}_n$ as long as $$\{m, m+2\}\subseteq [3, |p_1|-1]\cup [|p_2|+1, |p_3|-1]\cup\dots \cup [|p_{2k-2}|+1, n].$$
	In other words, the link of $\sigma$ in $\Lambda^{2k-1}_{n-2}=\lk(\{1,2\}, \Delta^{2k+1}_n)$ has at least $2(n-2-3(k-1))=2(n-3k+1)$ vertices. On the other hand, the number of such $\sigma$ is $2^{k-1} \binom{n-6k}{k-1}=\Omega(n^{k-1})$: indeed $\binom{n-6k}{k-1}$ is the number of ways of choosing a subset $\{|p_1|<|p_3|< \dots< |p_{2k-3}|\}$ of $[7, n-6]$ so that $|p_{2j-1}| -|p_{2j-3}|\geq 7$, while $2^{k-1}$ comes from the fact that we can attach a sign to any of the $k-1$ pairs $(|p_{2j-1}|, |p_{2j}|)$.
	\hfill$\square$\medskip

	In light of Lemmas \ref{lm: facet inclusion} and \ref{lm: facet}, we introduce the following definitions. We say that $1$ is {\em close} to every negative element of $V_n$ while $-1$ is {\em close} to every positive element. In addition, we call two elements of $V_n$ {\em close} to each other if their absolute values differ by at most $2$. Consider a set $F=\{i_1,i_2,\dots,i_{2j-1}, i_{2j}\} \subseteq V_{n}$, where $|i_1|<|i_2|<\dots< |i_{2j}|$. We say that $\{i_s, i_{s+1},\dots, i_r\}$ is a {\em run} of close elements in $F$ if every two consecutive elements of this segment are close to each other, but $i_{s-1}$ is not close to $i_s$ (or $s=1$) and $i_r$ is not close to $i_{r+1}$ (or $r=2j$). 
	
\smallskip\noindent {\it Proof of Lemma \ref{lemma:O(n^{k-2})}: \ }	
	Let $F=\{i_1,i_2,\dots,i_{2k-3}, i_{2k-2}\}$ be a $(2k-3)$-face of $\Delta^{2k-1}_{n-2}$, where $|i_1|<|i_2|<\dots< |i_{2k-2}|$. Then $F$ can be expressed in a unique way as a disjoint union of runs: $ F=R_1\cup\dots\cup R_q$. By Lemmas \ref{lm: facet inclusion} and \ref{lm: facet}, a facet containing the codimension $2$ face $F$ is a disjoint union of runs of even lengths. Hence only zero or two of $R_1, \dots, R_q$ have an odd length.
	
	{\bf Case 1:} two of the runs, say $R_a$ and $R_b$ (where $a<b$), have odd lengths. We prove that for a sufficiently large $n$, the link of any such $F$ computed in $\Delta^{2k-1}_{n-2}$ has less than $2(n-3k+1)$ vertices.
	
	i) If $a\neq 1$ or $a=1$, but $\pm 1\notin F$, then the link of such a face cannot have more than $8(2k-2)+2$ vertices: this is because if $G$ is a facet of $\Delta^{2k-1}_{n-2}$ that contains $F$, then $G\backslash F$ must contain 2 vertices each of which is close to some element of $F$; these could only be vertices $\pm 1$ along with $$\{v: |v|-|i_p|\in \{\pm 1, \pm 2\}, \;\text{for some} \;i_p\in R_a\cup R_b\} \subseteq \{v: |v|-|i_p|\in \{\pm 1, \pm 2\}, \;\text{for some} \;i_p\in F\}.$$
	
	ii) If $a=1$, and $R_1$ starts with $\{\pm 1, \pm 2\}$ or $\{\pm1, \pm3\}$, then the same argument as above applies and shows that the link has at most $8(2k-2)$ vertices.
	
	iii) If $a=1$, and $R_1$ starts with $\{1, -m\}$ for $m>2$, then the vertices of the link of $F$ are contained in $\{-\ell: 1<\ell\leq m-1 \}\cup \{v: |v|-|i_p|\in\{\pm1,\pm2\}, \;\text{for some}\; i_p\in F\}$, so there are at most $m+8(2k-2) < n+8(2k-2)$ such vertices and $n+8(2k-2)$ is smaller than $2(n-3k+1)$ assuming $n$ is large enough. The same argument works if $R_1$ starts with $\{-1,m\}$.
	
	iv) If $a=1$ and $R_1=\{1\}$, then the first element of $R_2$ is some positive number $m>3$. In this case as in case (iii), the vertices of the link of $F$ are contained in $\{-\ell: 1<\ell\leq m-1 \}\cup \{v: |v|-|i_p|\in\{\pm1,\pm 2\} \;\text{for some}\; i_p\in F\}$, so there are again at most $m+8(2k-2) \leq n+8(2k-2)$ such vertices. The same argument works if $R_1=\{-1\}$.
	
	{\bf Case 2: } all $R_1, \dots, R_q$ have even lengths. Assume that $G$ is a facet of $\Delta^{2k-1}_{n-2}$ containing $F$. By Lemma \ref{lm: facet inclusion}, the two elements of $G$ with the smallest absolute values form an edge in some $\Delta^1_m$. Hence there are the following three possible subcases; we consider them below.
	
	i) $\{i_1, i_2\}$ is not an edge of any $\Delta^1_m$. First, assume $|i_1|\neq 1$. Then, by Lemmas \ref{lm: facet inclusion} and \ref{lm: facet}, $G$ can only be of the following two types: 
	\begin{itemize}
		\item $F\cup\{i_{-1},i_0\}$, where $|i_{-1}|< |i_0|<|i_1|$;
		\item $F\cup \{u,v\}$, where $|u|=|1|$, $|i_1|-1$ or $|i_1|+1$, and $|v|-|i_p|\in\{\pm1,\pm 2\}$ for some $i_p\in F$.
	\end{itemize} 
	Thus, $f_0(\lk(F, \Delta^{2k-1}_{n-2}))\leq 2(|i_1|-1)+8(2k-2)$. Therefore, for this link to have at least $2(n-3k+1)$ vertices, we must have $|i_1|\geq n+10-11k$. There are at most $2^{2k-2}\binom{11k}{2k-2}$ such faces $F$.
	
	On the other hand, if $|i_1|=1$, then $G$ is of the form $F\cup\{u,v\}$, where $1<|u|<|i_2|$ and $|v|-|i_p|\in\{\pm 1,\pm 2\}$ for some $i_p\in F$. The same computation as above implies that there are at most $2^{2k-2}\binom{11k}{2k-3}$ such faces.
	
	ii) $\{i_1,i_2\}=\pm \{1, -m\}$ for some $m>1$. Say, $\{i_1,i_2\}=\{1, -m\}$. Let $G=F\cup \{u, v\}$, where $|u|<|v|$. We claim that then either $m\leq 5$, or $\{u, v\}$ and $\{3,4,\dots,m-3\}$ are disjoint. Indeed if $|u|< m$, then by Lemmas \ref{lm: facet inclusion} and \ref{lm: facet}, either $u=\pm 2$ or $-|m|<u<-2$, while $|v|-|i_p|\in\{\pm 1,\pm 2\}$ for some $i_p\in F$, $p\geq 2$. Thus, the link of $F$ will have at least  $2(n-3k+1)$ vertices only if $m\leq 5$ or $$2(n-3k+1)\leq 2(n-2)-|[3,m-3]|-|F\cup (-F)|=(2n-4)-(m-5)-(4k-4).$$ In either case, $m\leq 2k+3$. There are at most $2(2k+2)8^{k-2}\binom{n}{k-2}=O(n^{k-2})$ such faces $F$: the factor of $2(2k+2)$ counts the number of possible ways to choose the value of $m$ and the sign of  $\{i_1,i_2\}=\pm\{1,-m\}$, while $8^{k-2}\binom{n}{k-2}$ counts the number of ways to choose $|i_3|, |i_5|, \dots, |i_{2k-3}|$, the signs of $i_{2j-1}$ and the values of $i_{2j}$ (which must belong to $\{\pm(|i_{2j-1}|+1),\pm(|i_{2j-1}|+2)\}$).
	
	iii) $|i_2|=|i_1|+1$ and $i_1, i_2$ have the same signs. Assume that $G=F\cup \{i_{-1}, i_0\}$ is a facet of $\Delta^{2k-1}_{n-2}$, where $|i_{-1}|<|i_0|<|i_1|$. By Lemma \ref{lm: facet inclusion}, there exist $j\leq 2$ and $m$ such that 
	$$\{i_{-1}, i_0,\dots, i_3\}\in \pm \partial B^{5,j}_m=\pm \left(\partial B^{4, j}_{m-1}*m\right)\cup \pm \left(\partial B^{4, j-1}_{m-1}*m\right)  \cup \pm \left( B^{4,j}_{m-1}\backslash -B^{4, j-1}_{m-1}\right). $$ 
	Lemma \ref{lm: facets} along with the fact that $i_1$ and $i_2$ differ by 1 and have the same sign implies that if $$\{i_{-1}, i_0, i_1, i_2\}\in \pm \partial B^{4,2}_{m-1}\cup \pm B^{3,2}_{m-2}\subseteq \Delta^3_{m-1} \cup \Delta^3_{m-2},$$ then $ i_2\in \pm \{m-2, m-1\}$. Furthermore, in such a case we must have $\{i_{-1}, \dots, i_3\}\in \pm \big(\partial B^{4, 2}_{m-1}*m\big)\cup \pm B^{4,2}_{m-1}$ and so $|i_3|\leq |i_2|+2$. Otherwise, $\{i_{-1}, i_0, i_1, i_2\}$ must be in $\pm \partial B^{4,1}_{m-1}$ or $\pm B^{3,1}_{m-2}$, in which case it also follows that $|i_3|\leq |i_2|+2$.
	
	The above discussion shows that either the vertex set of the link of $F$ has no elements in $\{\pm2,\pm3,\dots,\pm(|i_1|-2)\}$ or $|i_3|\leq |i_2|+2$. In the former case, $$f_0\big(\lk(F, \Delta^{2k-1}_{n-2})\big) \leq 2(n-2-(|i_1|-3)-(2k-2)),$$ which means that this link has at least $2(n-3k+1)$ vertices only if $|i_1|\leq k+2$. But then the number of such faces $F$ is at most $2(k+2)\cdot 8^{k-2}\binom{n}{k-2} =O(n^{k-2})$. In the latter case, choosing $i_1$ determines $i_2$ and gives four possible choices for $i_3$ (as $i_3$ must belong to $\{\pm(|i_2|+1),\pm(|i_2|+2)\}$), which, by Lemma \ref{lm: facet}, in turn gives four possible choices for $i_4$. Therefore, the total number of such faces $F$ is also at most $2n \cdot 4^2\cdot 8^{k-3}\binom{n}{k-3}=O(n^{k-2})$.
\hfill$\square$\medskip


\section{Many cs $(2k-1)$-spheres that are cs-$(k-1)$-neighborly} \label{sec:(k-1)-neighb}
The goal of this section is to construct many cs $(2k-1)$-spheres with vertex set $V_n$ that are cs-$(k-1)$-neighborly. Our strategy is outlined at the end of Section \ref{sec:two-constructions}: it consists of starting with $\Delta^{2k-1}_n$ and symmetrically applying  bistellar flips. Showing that the resulting complexes are not pairwise isomorphic relies on understanding the set of automorphisms of $\Delta^{2k-1}_n$.
  
\begin{theorem}\label{thm: automorphism}
	Let $k\geq 2$ and let $n$ be sufficiently large. 
Then $\Delta^{2k-1}_n$ admits only two automorphisms: the identity map and the map induced by the involution $\alpha: i \mapsto -i$ for all $i\in V_n$.
\end{theorem}
\begin{proof}
We first treat the case of $k=2$. Let $\phi:V_n\to V_n$ be a bijection that induces an automorhism $\Phi$ of $\Delta^{3}_n$. By Corollary \ref{cor: edge links}, the edges $\pm\{1,2\}$ and $\pm\{n-1,n\}$ are the only edges whose links have length $2n-4$, hence $\Phi$ must map the set $A=\{\pm\{1,2\}, \pm\{n-1,n\}\}$ to itself. Similarly, $\Phi$ must also map the set $B=\{\pm\{n-2,n\}, \pm\{2,3\}, \pm\{n-3,n-1\}\}$ to itself because $B$ is the set of all edges whose links have length $2n-5$. Finally, for $\Phi$ to be an automorphism of $\Delta^{3}_n$, it must satisfy $\phi(-v)=-\phi(v)$ for all vertices $v$.

Now, observe that  $\phi(2)\neq \pm 1$, for otherwise $\Phi(\{2,3\})\notin B$. Similarly, $\phi(2)\neq \pm n$, or else $\phi(1)$ would be $n-1$ or $-(n-1)$, in which case,  $\{n-3,n-1\}$ would not be in $\Phi(B)$. The same argument shows that $\phi(2)\neq \pm (n-1)$, for otherwise $\phi(1)$ would be $n$ or $-n$, and so $\{n-2,n\}$ would not be in $\Phi(B)$. Thus $\phi$ must map $1$ to $1$ or $-1$ and $2$ to $2$ or $-2$; furthermore, the signs of $\phi(1)$ and $\phi(2)$ must be the same. Assume first that $\phi(1)=1$, $\phi(2)=2$. Then $\Phi(\{2,3\})\in B$ implies that $\Phi(\{2,3\})=\{2,3\}$, so $\phi(3)=3$. In addition, $\Phi(\{n-1,n\}) \in A$ and hence $\Phi(\{n-1,n\})\in \{\pm\{n-1,n\} \}$. Since the link of both $\{n-1,n\}$ and $\{-n+1,-n\}$ is the cycle $(1,2,3,\dots, n-2,-1,\dots)$, and since $\phi$ is the identity on $\{1,2,3\}$, we conclude that $\phi$ must be the identity on $V_{n-2}$. Finally, since the link of $\{1,2\}$ is the cycle $(3,5,\dots,n-3,n-1,n, n-2,\dots)$ or the cycle $(3,5,\dots, n-2, n, n-1, n-3,\dots)$ depending on the parity of $n$, it follows that $\phi$ must also be the identity on $\{n-1, n\}$. Exactly the same argument applies if $\phi(1)=-1$, $\phi(2)=-2$ and shows that in this case $\phi$ is  the involution $\alpha$.
	
	For $k\geq 3$, we use induction on $k$. By Theorem \ref{lm: edge link main result}, the complex $\Delta^{2k-1}_n$ has exactly four edges whose links are both cs and cs-$(k-1)$-neighborly; they are $\pm \{1, 2\}$ and $\pm \{n-1,n\}$. Furthermore, by Theorem \ref{thm: distinct edge links} we could identify which two of these four edges are $\pm\{n-1, n\}$; their links are $\Delta^{2k-3}_{n-2}$. By induction, there are only two automorphisms of $\Delta^{2k-3}_{n-2}$: the one induced by the identity on $V_{n-2}$ and another one induced by the involution $\alpha$. In other words, there are only two ways to label the vertices of the link of $\{n-1,n\}$ by the elements of $V_{n-2}$. Since by Lemma \ref{lem:F_i-G_i-flip}, the link of $F=\{n-4k+3, n-4k+6, n-4k+10, \dots, n-2\}$ contains the vertex $n$ but does not contain any of the vertices $\pm(n-1), - n$, once such a labeling is chosen, the vertex $n$ is determined uniquely, and hence so is $-n$. Finally, since $\lk(\{n-1, n\}, \Delta^{2k-1}_n)$ is cs-$(k-1)$-neighborly while $\lk(\{n-1, -n\}, \Delta^{2k-1}_n)$ is not, this also uniquely determines the vertices $n-1$ and $-n+1$.
\end{proof}

Let $I=[3, n-4k+2]$. 
In Section \ref{sec:two-constructions}, for $i\in I$, we defined the sets $F_i:=\{i, i+3, i+7, i+11, \dots, i+4k-5\}$ and $G_i:=\{i-1, i+1, i+5, i+9, \dots, i+4k-3\}$. We then proved in Lemma~\ref{lem:F_i-G_i-flip} that 
replacing $\pm\st(F_i,\Delta^{2k-1}_n)= \pm\big(\overline{F_i} * \partial\overline{G_i}\big)$ with $\pm\big(\partial\overline{F_i}*\overline{G_i}\big)$ constitutes admissible bistellar flips and results in new cs combinatorial spheres. Further, since the stars $\pm\st(F_i,\Delta^{2k-1}_n)$ are pairwise disjoint, we can simultaneously perform any number of such flips. With Theorem \ref{thm: automorphism} in hand, we are ready to show that distinct combinations of such symmetric flips produce distinct cs combinatorial $(2k-1)$-spheres that are cs-$(k-1)$-neighborly.

\begin{theorem} \label{thm:cs-(k-1)-neighb-constr}
	Let $k\geq 3$ and let $n$ be sufficiently large. There are at least $\Omega(2^n)$ non-isomorphic cs combinatorial $(2k-1)$-spheres on vertex set $V_n$ that are cs-$(k-1)$-neighborly.
\end{theorem}
\begin{proof}
Let $J$ be any subset of $I=[3, n-4k+2]$. By Lemma \ref{lem:F_i-G_i-flip}, 
we can simultaneously apply bistellar flips replacing $\pm\st(F_i,\Delta^{2k-1}_n)=\pm(\overline{F_i}*\partial\overline{G_i})$ with $\pm(\partial\overline{F_i}*\overline{G_i})$ for for all $i\in J$. The resulting complex $\Gamma(J)$ is a combinatorial sphere with missing $(k-1)$-faces $F_i$ for $i\in J$, but with the same $(k-2)$-skeleton as $\Delta^{2k-1}_n$. Hence $\Gamma(J)$ is cs-$(k-1)$-neighborly. Observe that since $1$ and $n$ are not in any of $F_i, G_i$, the link $\lk(\{n-1, n\}, \Gamma(J))=\lk(\{n-1, n\},\Delta^{2k-1}_n)$ and $\lk(\{1,2\}, \Gamma(J))=\lk(\{1,2\}, \Delta^{2k-1}_n)$. By Theorem \ref{thm: distinct edge links} and the fact that $\skel_{k-2}(\Gamma(J))=\skel_{k-2}(\Delta^{2k-1}_n)$, the only edges in $\Gamma(J)$ whose links are both cs and cs-$(k-1)$-neighborly  are $\pm \{1,2\}$ and $\pm \{n-1,n\}$. As in the proof of Theorem~\ref{thm: distinct edge links}, we thus can identify which two edges are $\pm \{n-1, n\}$. We then can use Theorem~\ref{thm: automorphism} to determine the labels of all the vertices in $\lk(\pm \{n-1, n\}, \Delta^{2k-1}_n)$ up to the involution $\alpha$.  As in the proof of Theorem \ref{thm: automorphism}, this in turn determines the vertices $\pm (n-1), \pm n$. (Note that the face $F$ of $\Delta^{2k-1}_n$ used at the end of the proof of Theorem 6.1 is $F_{n-4k+3}$. Since $n-4k+3\notin I$, the face $F$ and its link are unaffected by bistellar flips.) Therefore, the  spheres $\Gamma(J)$ (for $J\subseteq I$) are pairwise non-isomorphic. The result  follows since  there are $2^{|I|}=\Omega(2^n)$ of them.
\end{proof}

\section{Many cs $3$-spheres that are cs-$2$-neighborly} \label{sec:many-3-spheres}

In this section we turn our attention to the $3$-dimensional case and show that there exist many non-isomorphic cs $3$-spheres on $V_{n+1}$ that are cs-$2$-neighborly, see Theorems  \ref{thm:Delta(I)} and \ref{thm: 2^n combinatorial types}. As in \cite{Jockusch95}, our construction is based on Lemma \ref{lm: induction method}: (i) start with $\Delta^3_{n}$, a cs $3$-sphere on $V_{n}$ that is cs-$2$-neighborly; (ii) in $\Delta^3_{n}$, find a $3$-ball $B$ such that $B$ is cs-$1$-neighborly (w.r.t.~$V_n$) and $1$-stacked, and shares no common facets with $-B$; then (iii) replace $B$ and $-B$ with the cones over their boundaries. Our first goal is to construct many non-isomorphic balls with these properties. 

For brevity, write a facet $\{a,b,c,d\}\in \Delta^3_n$ as $abcd$. Consider Figure \ref{figure: tree}. By Corollary \ref{cor: edge links} and the fact that $B^{3,1}_n\subseteq \Delta^{3}_n$, each node in Figure \ref{figure: tree} corresponds to a facet and each row represents a subset of facets in $\lk(\{i, i+2\}, \Delta^3_n)$ for $3\leq i\leq n-2$ or $\lk(\{n-1,n\}, \Delta^3_n)$. We refer to the middle row --- the row corresponding to $\lk(\{n-1,n\}, \Delta^3_n)$ --- as row $0$, and to the row corresponding to $\lk(\{n-i, n-i+2\}, \Delta^3_n)$ as row $i-1$ (here, $2\leq i\leq n-3$). Thus, the two rows adjacent to the middle row are rows $1$ and $2$. More generally, the distance from row $i$ to the middle row is $\lfloor (i+1)/2\rfloor$.

We start by defining a family of trees as subgraphs of the facet-ridge graph of $\Delta^3_n$. We will then prove that these trees are facet-ridge graphs of pairwise non-isomorphic cs-$1$-neighborly and $1$-stacked balls in $\Delta^3_n$.
\begin{figure}
\includegraphics[scale=0.9]{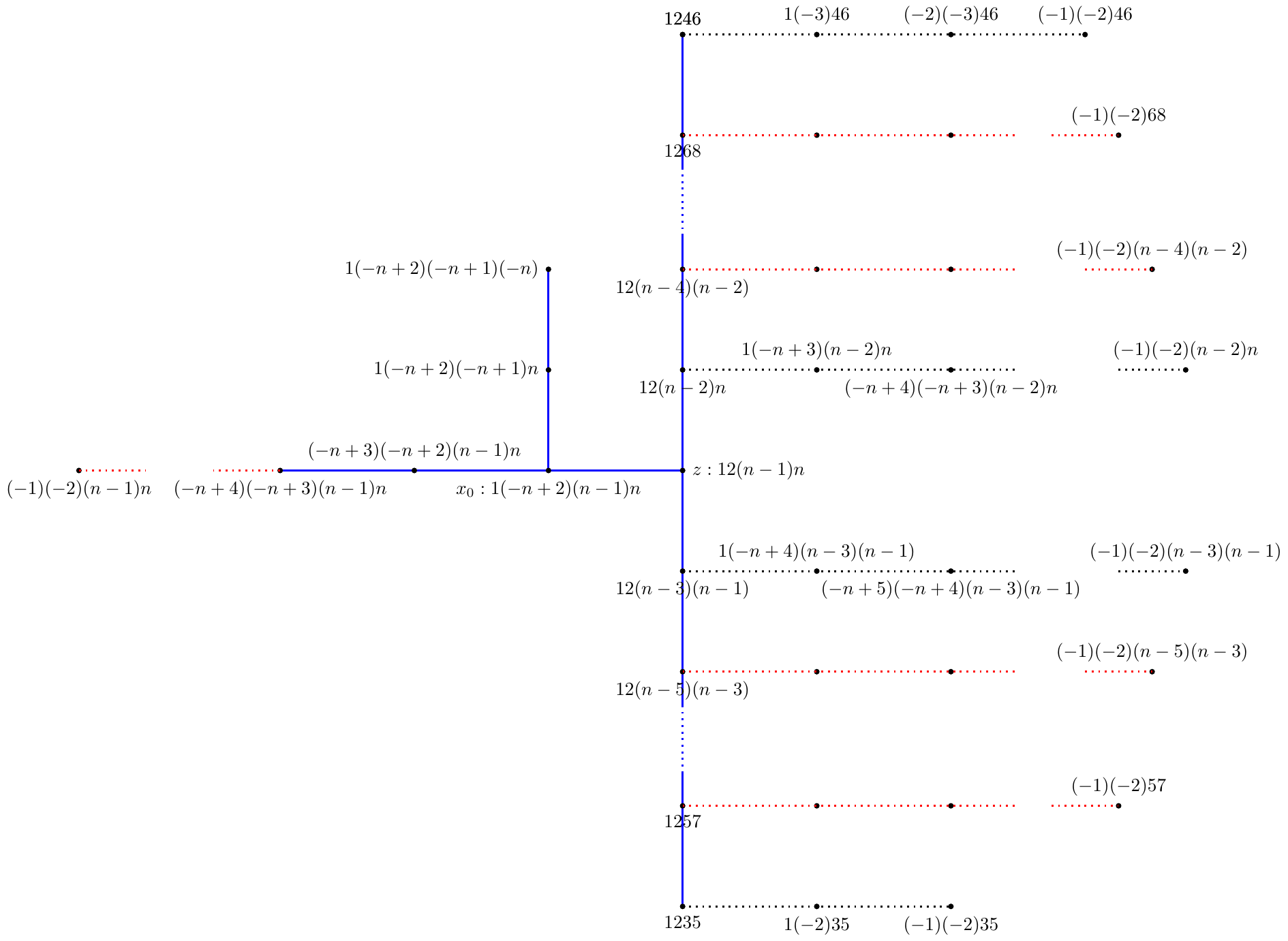}
\caption{A subgraph of the facet-ridge graph of $\Delta^{3}_n$ when $n$ is even}\label{figure: tree}
\end{figure}

\begin{definition}
	Let $n\geq 10$. Consider the following collection of subsets
	\begin{equation}\label{eq: I}
	\mathcal{I}_n:=\left\{I=\{i_1<i_2<\dots <i_p\}\subseteq [3, n-6]: i_2-i_1>1\;\;\mathrm{if}\;\; p\geq 2\right\}.
	\end{equation}
For each $I\in \mathcal{I}_n$, define the tree $T(I)$ as the subgraph of the graph in Figure \ref{figure: tree} that consists of the two blue vertical paths in the picture, i.e., the path from node $1246$ to node $1235$ and the path corresponding to the facet-ridge graph of $(1, -n+2)*(n-1, n, -n+1, -n)$, together with several horizontal paths attached to the middle vertical path and specified as follows: the horizontal paths are contained in the rows indexed by the elements of $I\cup \{0\}=\{i_0=0<i_1<\cdots <i_p\}$, and have lengths $i_1-i_0, i_2-i_1, \ldots, i_{p+1}-i_p$, respectively, where $i_{p+1}:=n-2$. In more detail,
\begin{itemize}
	\item The path in row $i_0$ starts from the node $12(n-1)n$ and goes through nodes $1(-n+2)(n-1)n$, $(-n+3)(-n+2)(n-1)n$, etc.~until it reaches the node $(-n+i_1+1)(-n+i_1)(n-1)n$.
	\item For $1\leq j\leq p$, the path in row $i_j$ starts from the node $12(n-i_j-1)(n-i_j+1)$, moves to the node $1(-n+i_j+2)(n-i_j-1)(n-i_j+1)$ and continues all the way until it reaches the node $(-n+i_{j+1}+1)(-n+i_{j+1}) (n-i_j-1)(n-i_j+1)$.
\end{itemize}

We also define $T(\mathcal{I}_n):=\{T(I) : I \in \mathcal{I}_n\}$. 
\end{definition}

\noindent For instance, if $I=\emptyset$, then $T(I)$ consists of the blue parts of Figure \ref{figure: tree} plus the entire middle row. Note also that the condition $I\subseteq [3,n-6]$ guarantees that all four black horizontal paths in Figure~\ref{figure: tree} are excluded from $T(I)$. On the other hand, the horizontal path in the middle row always has at least three edges. In short, the blue part of Figure~\ref{figure: tree} is contained in {\em all} trees of $T(\mathcal{I}_n)$. Finally, it is worth remarking that $T(I)$ has exactly $2n-3$ nodes: $n-3$ nodes come from the middle column, $3$ more from the short blue vertical line and $(i_1-i_0)+\cdots+(i_{p+1}-i_p)-1=n-3$ additional nodes come from the horizontal paths.

\begin{lemma}\label{lm: I and T(I)}
	Let $n\geq 10$. The trees in $T(\I_n)$ are pairwise non-isomorphic. 
\end{lemma} 
\begin{proof}
	It suffices to show that given an unlabeled tree $T$ in $T(\I_n)$, we can reconstruct the set $I$ such that $T=T(I)$. 
	Let $N$ be the set of nodes of $T$ of degree $3$. Call an element $u$ of $N$ {\em interior} if there exist $v,w\in N\backslash \{u\}$ such that the unique path from $v$ to $w$ in $T$ contains $u$; call $u$ {\em exterior} otherwise. Note that $N$ has at most three exterior elements and that one of them is the node that corresponds to the facet $1(-n+2)(n-1)n$; we denote this node by $x_0$. Note also that $|N|=|I|+2$. In particular, $p=|I|$ is determined from $T$, so we can assume in the rest of the proof that $p>0$.
	
	Suppose first that  $N$ has exactly three exterior elements: $x_0$ and two more; these two nodes lie on the middle vertical column of Figure \ref{figure: tree} (the column that corresponds to the link of $\{1,2\}$); we denote them by $x_1$ and $x_2$. Also, denote by $z$ the node corresponding to the facet $12(n-1)n$. We start by showing how to identify $x_0$ and $z$ in the unlabeled tree $T$. To identify $z$, for each two of the three exterior elements of $N$, consider the unique path in $T$ that connects them. Then $z$ is precisely the intersection of these three paths. Further, the distance in $T$ from $z$ to  $x_0$ is $1$, while the conditions $i_1\geq 3, i_2\geq 5$ guarantee that the distance from $x_1$ to $z$ is at least $2$ and so is the distance from $x_2$ to $z$. This determines $x_0$.
	
	We now discuss how to reconstruct the set $I$ from $T$.  
	Observe that there are exactly two leaves in $T$ with the property that the path from $x_0$ to each of them does not contain any nodes in $N\backslash x_0$: they are the leaf of the short vertical blue line, i.e., the node corresponding to $1(-n+2)(-n+1)(-n)$, and the leaf of $T$ that lies in the middle row, i.e., the node corresponding to $(-n+i_1+1)(-n+i_1)(n-1)n$. These two leaves are at distance $2$ and $i_1-1$ from $x_0$. 
	This determines $i_1$. Let $y_1$ be the closest to $z$ (after $x_0$) node of degree 3. Since $i_2-i_1>1$, such $y_1$ is unique. By construction,  $y_1$ must be in row $i_1$, and so $i_1$ equals either $2\dist(y_1, z)-1$ or $2\dist(y_1, z)$, depending on the parity of $i_1$. Since we already know $i_1$, this allows us to determine which side of $z$ is ``even" and which side is ``odd". 
	Then, for any node $y\in N\backslash\{x_0,y_1,z\}$, the distance from $y$ to $z$ combined with whether $y$ is on the same side of $z$ as $y_1$ allows us to determine $y$'s row number. The set of these row numbers together with $i_1$ is the set $I$.

	It is left to consider the case where $N$ has only two exterior elements $x_0$ and $x_1$. In this case, the path connecting $x_0$ to $x_1$ in $T$ contains all elements of $N$. List the elements of $N$ in the order we encounter them on this path: $x_0, z, y_1,\ldots,y_{p-1}, y_p=x_1$. 
	Our first task is to show that we can distinguish between $x_0$ and $x_1$. If $p=1$, this is easy as $x_0$ and $z$ are neighbors in $T$ while $x_1$ and $z$ are not. So assume $p>1$. Observe that there is exactly one leaf in $T$, denote it by $\tilde{z}$, such that the path from $z$ to  $\tilde{z}$ does not contain any nodes in $N\backslash z$: it is one of the two leaves of the middle column. In particular, $\dist(z,\tilde{z})$ is a fixed number $\lfloor n/2 \rfloor-2$. Similarly, there is exactly one leaf $\tilde{y}_{p-1}$ such that the path from $y_{p-1}$ to  $\tilde{y}_{p-1}$ does not contain any nodes in $N\backslash y_{p-1}$: it is the leaf of $T$ in row $i_{p-1}$.  Since $x_0$ and $z$ are neighbors in $T$, to distinguish between $x_0$ and $x_1$, it suffices to show that if $y_{p-1}$ is a neighbor of $x_1$ in $T$, then $\dist(y_{p-1},\tilde{y}_{p-1})\neq \lfloor n/2 \rfloor -2$. Indeed, 
if $y_{p-1}$ is a neighbor of $x_1=y_{p}$, then the rows $i_{p-1}$ and $i_p$ are adjacent rows, and so $\dist(y_{p-1},\tilde{y}_{p-1})=i_p-i_{p-1}=2\neq \lfloor n/2 \rfloor -2$, for $n\geq 10$, as desired. 

Once we can distinguish between $x_0$ and $x_1$, we can also determine $z$: it is the only neighbor of $x_0$ that has degree $3$. We can now reconstruct $I$ exactly in the same way as in the case of three exterior elements.
\end{proof}  

We now come to the two main definitions of this section (see, Definitions \ref{def:B(I)} and \ref{def:Delta(I)}):
\begin{definition} \label{def:B(I)}
	Let $n\geq 10$. For $I\in\mathcal{I}_n$, define $B(I)$ as the pure simplicial complex whose set of facets is given by the set of nodes of $T(I)$.
\end{definition} 

The following lemma describes some important properties of $B(I)$. Recall that a pure $d$-dimensional simplicial complex $\Delta$ is {\em shellable} if the facets of $\Delta$ can be ordered $F_1, \dots, F_m$ in such a way that for every $2\leq k\leq m$, $\overline{F_k}\cap(\cup_{i<k} \overline{F_i})$ is pure $(d-1)$-dimensional. The order $(F_1, \dots, F_m)$ is called a {\em shelling order} of $\Delta$. It is known that an ordering $(F_1,\dots,F_m)$ of facets is a shelling if and only if for every $k\geq 2$, the collection of faces of $\overline{F_k}\backslash (\cup_{i<k} \overline{F_i})$ (i.e., the collection of new faces of $\overline{F_k}$) has a unique minimal face w.r.t~inclusion; this face is called the {\em restriction face} of $F_k$. (See \cite[Chapter 8]{Ziegler} for more background on shellings.)

\begin{lemma}\label{lm: property of B(I)}
	Let $n\geq 10$. The complex $B(I)$ is a cs-$1$-neighborly (w.r.t.~$V_n$) and  $1$-stacked combinatorial $3$-ball in $\Delta^{3}_n$; furthermore, $B(I)$ and $-B(I)$ share no common facets.
\end{lemma}
\begin{proof}
	If $I\in \mathcal{I}_n$, then by Corollary \ref{cor: edge links}, $B(I)$ is a subcomplex of $\Delta^{3}_n$.  Our construction shows that each vertex of $\{1,2,\ldots,n, -n, -n+1, -n+2\}$ appears in at least one of the blue vertical paths of $T(I)$ while each vertex of $\{-1,\ldots,-n+3\}$ appears in at least one row of $T(I)$. Hence $B(I)$ is cs-$1$-neighborly w.r.t.~$V_n$.  Since $B(I)$ has $2n-3$ facets and $2n$ vertices, it follows that any ordering of the nodes of $T(I)$ in a manner that keeps the so-far-constructed subgraph of $T(I)$ connected throughout the process gives a shelling order of $B(I)$, with all restriction faces being of size $1$. This proves that $B(I)$ is a $1$-stacked combinatorial 3-ball. Finally, the fact that $B(I)$ and $-B(I)$ do not share common facets is easy to see from the definition of $B(I)$: no two nodes of the entire Figure~\ref{figure: tree} (let alone its subgraph $T(I)$) correspond to antipodal facets.
\end{proof}

\begin{definition} \label{def:Delta(I)}
	Let $n\geq 10$ and $I\in \I_n$. Define $\Delta(I)$ as the complex obtained from $\Delta^{3}_n$ by replacing $\pm B(I)$ with $\pm (\partial B(I)*(n+1))$.
\end{definition}
Lemma \ref{lm: induction method} along with Lemma \ref{lm: property of B(I)} imply the following
\begin{theorem}  \label{thm:Delta(I)}
	For $n\geq 10$ and $I\in \mathcal{I}_n$, the complex $\Delta(I)$ is a cs combinatorial $3$-sphere on $V_{n+1}$ that is cs-$2$-neighborly.
\end{theorem}

Our next goal is to prove that if $I,J\in \mathcal{I}_n$ and $I\neq J$, then the spheres $\Delta(I)$ and $\Delta(J)$ are non-isomorphic.  The proof will require the following two lemmas.

\begin{lemma}\label{lm: large link 23}
	Let $n\geq 10$, let $I\in \mathcal{I}_n$, and let $e$ be an edge of $\Delta(I)$. Then \begin{enumerate}
		\item $f_0(\lk(e, \Delta(I)))\geq 2n-3$ if and only if $e=\pm \{2,3\}$.
		\item If $e=\{2,i\}$ for $i\in V_n\backslash \pm 2$, then $f_0(\lk(e, \Delta(I))) \mbox{ is }\begin{cases}
		n+2 & i=1\\
		n-1 & i=n+1\\
		<n-1 & i\neq 1,3,n+1.
		\end{cases}$
		\item $f_0(\lk(e, \Delta(I)))=2n-6$ if $e=\{3,4\}$.
	\end{enumerate}
\end{lemma}
\begin{proof} We first prove parts 1 and 2. Since $\Delta(I)$ is cs, it suffices to prove part 1 only for one edge from each pair $\pm e$.	
Assume that $\pm (n+1)\notin e$. Then the link $\lk(e, \Delta(I))$ is obtained from the graph-theoretic cycle $\lk(e, \Delta^{3}_n)$ by replacing two paths, namely, $\lk(e, B(I))$ and $\lk(e,-B(I))$, with the cones over their boundaries, i.e., with $\partial \lk(e, B(I)) *(n+1)$ and $\partial \lk(e, -B(I)) *(-n-1)$, respectively. In other words, the number of vertices in the resulting complex is the number of vertices in the original link minus the sum of the lengths of the two paths plus zero, two, or four depending on whether zero, one or two of these paths are non-empty. In particular, $f_0(\lk(e,\Delta(I)))\leq f_0(\lk(e,\Delta^{3}_n))+2$. Thus,
if $f_0(\lk(e, \Delta(I)))\geq 2n-3$, then we must be in one of the following two cases:
	
	\smallskip\noindent{\bf Case 1:} $f_0(\lk(e, \Delta^{3}_n))=2n-4$. 
	Then by Corollary \ref{cor: edge links}, 
	$e=\pm \{1,2\}$ or $\pm\{n-1, n\}$. However, by definition of $B(I)$ (see Figure \ref{figure: tree}), $f_1(\lk(\{1,2\}, B(I)))=n-3$ and $f_1(\lk(\{1,2\}, -B(I)))=1$, and so $$f_0(\lk(\{1,2\}, \Delta(I)))=(2n-4)-(n-3)-1+4=n+2.$$
	Also, $f_1(\lk(\{n-1, n\}, B(I)))=i_1+1> 3$ and $f_1(\lk(\{-n+1, -n\}, B(I)))=1$. Hence a similar computation shows that for $e= \{n-1,n\}$, $f_0(\lk(e, \Delta(I)))< 2n-3$.
	
	\smallskip\noindent{\bf Case 2:} $f_0(\lk(e, \Delta^{3}_n))=2n-5$. 
	Then by Corollary \ref{cor: edge links}, $e=\pm \{2,3\}$, $\pm\{n-3, n-1\}$ or $\pm\{n-2,n\}$. However, $\lk(e, -B(I))=\emptyset$ for $e=\{n-3, n-1\}$ or $\{n-2,n\}$, and hence for each of these edges, $f_0(\lk(e,\Delta(I))= (2n-5)-1+2<2n-3$. On the other hand,  the links of $\{2,3\}$ in $B(I)$ and $-B(I)$ are both $1$-simplices, and so $f_0(\lk(e, \Delta(I)))=(2n-5)-2+4=2n-3$.  
	
	\smallskip\noindent Therefore, for edges $e\subseteq V_n$, the link $\lk(e, \Delta(I))$ has $2n-3$ vertices if and only if $e=\pm \{2,3\}$.
	
	Now, if $e=\{i, n+1\}$, then $f_0(\lk(e, \Delta(I)))=f_0(\lk(i, \partial B(I)))= f_0(\lk(i, B(I)))$. By definition of $B(I)$,  $f_0(\lk(2, B(I)))=n-1$. Indeed, the vertex set of $\lk(2,B(I))$ consists of all vertices (but $2$ itself) that appear in the nodes of the middle column of Figure \ref{figure: tree}, i.e., it is  $\{1\}\cup [3,n]$. This shows that $f_0(\lk(\{2, n+1\}, \Delta(I)))=n-1$. To see that other edges containing $n+1$ have ``short" links observe that the vertex set of $\lk(1,B(I))$ does not contain $-n+3, -n+4$. For any $i\in[3,n]$, the vertex set of $\lk(i, B(I))$ contains at most six positive vertices, namely, $\{1,2, i-2,i-1,i+1,i+2\}\cap [n]$. Finally, for any negative vertex $i\in V_n$, the vertex set of $\lk(i,B(I))$ misses $3$ and $4$. Hence we conclude that for any $i\in V_n$, $f_0(\lk(\{i, n+1\}, \Delta(I)))<2n-3$. This completes the proof of part 1.

	To finish the proof of part 2, note that if $i\in V_n\backslash\{1,3\}$, then  by Corollary \ref{cor: edge links},
	$$f_0(\lk(\{2, i\}, \Delta(I)))\leq f_0(\lk(\{2, i\}, \Delta^{3}_n))+2<n-1.$$
	Also $f_0(\lk(\{2,-n-1\}, \Delta(I)))=f_0(\lk(2, -\partial B(I))=f_0(\lk(-2, B(I))=4$. The other cases were already discussed above.
	
	Finally, part 3 follows from the fact that $f_0(\lk(\{3,4\}, \Delta^{3}_n))=2n-7$, $\lk(\{3,4\}, B(I))=\emptyset$, and $\lk(\{3,4\}, -B(I))$ is a 1-simplex.
\end{proof}

\begin{lemma}\label{lm: link of n}
	Let $n\geq 10$ and let $I, J\in \mathcal{I}_n$. If $\Delta(I)$ and $\Delta(J)$ are isomorphic, then $\lk(n+1, \Delta(I))$ and $\lk(n+1, \Delta(J))$ are also isomorphic.
\end{lemma}
\begin{proof}
To prove the lemma, it suffices to show that given $\Delta(I)$ (for $I\in \mathcal{I}_n$), we can uniquely identify  the vertices $\pm(n+1)$ of $\Delta(I)$. By Lemma \ref{lm: large link 23}(1), the edges $\pm \{2,3\}$ are the only edges of $\Delta(I)$ whose links have at least $2n-3$ vertices. This allows us to identify the edge $\{2,3\}$ of $\Delta(I)$ (up to the involution $\alpha$). However, by Lemma \ref{lm: large link 23}(2) and (3), the link of $\{2,i\}$ in $\Delta(I)$ has no more than $n+2$ vertices if $i\neq 3$, while the link of $\{3,4\}$ has $2n-6$ vertices. This allows us to distinguish between the vertices $2$ and $3$, and hence to determine the vertex $2$ (up to the involution). Finally, Lemma~\ref{lm: large link 23}(2) implies that there is a unique vertex link $\lk(i, \lk(2, \Delta(I)))$ in $\lk(2, \Delta(I))$ with exactly $n-1$ vertices: it is the link of $i=n+1$. This determines the vertex $n+1$ (up to the involution), and yields the result.
\end{proof}

Our discussion in this section culminates with the following result:
\begin{theorem}\label{thm: 2^n combinatorial types}
	Let $n\geq 10$ and let  $\Delta(\mathcal{I}_n)=\{\Delta(I): I\in \mathcal{I}_n\}$. The complexes in $\Delta(\mathcal{I}_n)$ are pairwise non-isomorphic. In particular, there are $\Omega(2^n)$ non-isomorphic cs combinatorial 3-spheres on $V_{n+1}$ that are cs-2-neighborly. 
\end{theorem}
\begin{proof}
	Let $I, J\in \mathcal{I}_n$. By Lemma \ref{lm: link of n},  for $\Delta(I)$ and $\Delta(J)$ to be isomorphic, the links of $n+1$ in $\Delta(I)$ and $\Delta(J)$ must be isomorphic. Since these two links, namely, $\partial B(I)$ and $\partial B(J)$, are stacked $2$-spheres, we conclude that the facet-ridge graphs $T(I)$ and $T(J)$ of their associated stacked balls are isomorphic. Hence, by Lemma \ref{lm: I and T(I)}, $I$ and $J$ must be the same set. The last claim follows from the fact that the size of the collection $\mathcal{I}_n$ defined in (\ref{eq: I}) is $\geq 2^{n-9}$ and from Theorem \ref{thm:Delta(I)}.
\end{proof}

\section{The sphere $\Delta^3_n$ is shellable} \label{sec:shellable}
In \cite[Problem 5.1]{N-Z} it was asked whether the combinatorial $(2k-1)$-spheres $\Delta^{2k-1}_n$ are shellable for all $k\geq 2$ and $n\geq2k$. Here we answer this  question in the $3$-dimensional case:
we verify that the spheres $\Delta^3_n$ are  shellable, by showing that they possess a symmetric shelling. A shelling order of a cs simplicial complex is called {\em symmetric} if it is the form $(F_1, F_2,\dots, F_m, -F_m, -F_{m-1},\dots,-F_1)$.
\begin{theorem} \label{thm:shellable}
	Let $n\geq 4$. There exists a symmetric shelling order of $\Delta^{3}_n$.
\end{theorem}
\begin{proof}
Our strategy is as follows: use equation (\ref{eq:decomposition}) to separate one half of the facets of $\Delta^{3}_n$ in the $n-2$ blocks described below; list these blocks in the shelling we are about to construct in the following order:
\begin{itemize}
\item The facets of $B^{3,1}_{n}$.
\item The facets of $\partial B^{3,1}_{k-1}*k$ that are not in $\pm B^{3,1}_{k}$. Here $5\leq k\leq n$. These $n-4$ blocks will be listed in the decreasing order of $k$, i.e., from $k=n$ to $k=5$;
\item Three of the six facets of $\Delta^{3}_4\backslash\pm B^{3,1}_4$ (which three will be specified later).
\end{itemize}
We will now discuss the order inside each of these blocks. Then we will list the other half of the facets to make the ordering symmetric.

	The ball $B^{3,1}_n$ is a stacked ball, hence shellable. We list its facets in any shelling order of $B^{3,1}_n$. Now, for $k=n, n-1, \dots, 5$, the facets of $\partial B^{3,1}_{k-1}*k$ that are not in $\pm B^{3,1}_{k}$ are described in Lemma~\ref{lm: facets}: they consist of $$F_{k,1}:=(-k+3)(-k+2)(-k+1)k, \quad F_{k,2}:=1(-k+3)(-k+1)k,$$ and the facets of  $$(k-3,k-4,\dots,1,-(k-3), \dots, -2,-1)*(k-2,k).$$ We order this block by starting with $F_{k,1}, F_{k,2}$ followed by the rest of the facets in the order we encounter them when moving along the path  $B^{1,1}_{k-3}=(k-3,k-4,\dots,1,-(k-3), \dots, -2,-1)$.

	To show that this is a partial shelling order, consider the facet $F_{k,1}=(-k+3)(-k+2)(-k+1)k$ and its $2$-faces. The face $\{-k+3, -k+2, k\}$ is contained in the preceding facet $(-k+3)(-k+2)k( k+2)\in \partial B^{3,1}_{k+1}*(k+2)$ if $k\leq n-2$, in $(-n+4)(-n+3)(n-1)n \in B^{3,1}_n$ if $k=n-1$, and in $(-n+3)(-n+2)(n-1)n\in B^{3,1}_n$ if $k=n$. Also,  $\{-k+2, -k+1, k\}$ is contained in the preceding facet $(-k+2)(-k+1)k(k+2)\in \partial B^{3,1}_{k+1}*(k+2)$ if $k\leq n-2$, in $(-n+3)(-n+2)(n-1)n\in B^{3,1}_n$ if $k=n-1$, and in $1(-n+2)(-n+1)n \in B^{3,1}_n$ if $k=n$. On the other hand, the edge $\{-k+3,-k+1\}$ is not contained in any of the earlier facets, and so this edge is the unique minimal new face of $F_{k,1}$. Similarly, for the facet $F_{k,2}=1(-k+3)(-k+1)k$, its $2$-face $\{1, -k+1, k\}$ is contained in the preceding facet $1(-k+1)k(k+2)$ if $k\leq n-2$, in $1(-n+2)(n-1)n$ if $k=n-1$, and in $1( -n+2)(-n+1)n$ if $k=n$. Also
	$\{-k+3, -k+1, k\}$ is contained in  $F_{k,1}$. As $\{1,-k+3\}$ is not contained in any of the earlier facets, it is the unique minimal face of $F_{k,2}$.
	
	The rest of the facets in this block are of the form $ij(k-2)k$, where $\{i, j\}$ is in the path $B^{1,1}_{k-3}=(k-3,k-4,\dots, -2,-1)$. Assume $i$ is to the left of $j$ in this path. Then $\{i, j, k\}$ is contained in the preceding facet $ijk(k+2)$ if $k\leq n-2$ and it is contained in $ij(n-1)n$ if $k=n-1$ or $n$. Also $\{i, k-2, k\}$ is a face of the immediately preceding facet in the order if $i\neq k-3$, of $(k-3)(k-2)k(k+2)$ if $i=k-3$ and $k\leq n-2$, and of a facet of $B^{3,1}_n$ if $i=k-3$ and $k=n-1$ or $n$. Hence the unique minimal new face of $ij(k-2)k$ is $\{j, k-2\}$. 
	
	Finally, we make the very last block consist of the facets $(-1)(-2)(-3)4$, $1(-2)3(-4)$, $123(-4)$ in this order. Using the same argument as we used for $F_{k,1}=(-k+3)(-k+2)(-k+1)k$, we see that the unique minimal new face of $(-1)(-2)(-3)4$ is $\{-1, -3\}$. Since $\{1, -2,-4\}\subset 1(-2) (-4)5\in \partial B^{3,1}_4*5$ and $\{1, -2, 3\}\subset 1(-2)35\in \partial B^{3,1}_4*5$ while $\{3,-4\}$ is not contained in any of the earlier facets, it follows that the unique minimal new face of $1(-2)3(-4)$ is $\{3, -4\}$. Similarly, the unique minimal new face of $123(-4)$ is $\{2, -4\}$. Thus the order described above is indeed a shelling order of the subcomplex formed by half of the facets of $\Delta^{3}_n$; furthermore, no two of the facets in this half are antipodal, so we can (uniquely) complete this order to a symmetric order of all facets of $\Delta^3_n$. This finishes the proof because by symmetry, for any facet $M$ in the second half of this order, the unique minimal new face of $M$ will be $M\backslash\tau$ where $-\tau$ is the unique minimal new face of the facet $-M$, which is in the first half of the order.
\end{proof}	

\begin{remark}  \label{rem:Murai-Nevo}
	In \cite[Question 6.5]{MuraiNevo2013}, Murai and Nevo asked whether there exists a $2$-stacked combinatorial $d$-ball $B$ such that $B$ is shellable but its boundary complex $\partial B$ is not polytopal. We show that the complex  $B^{4,2}_{6}$ provides an affirmative answer to their question. By definition, $B^{4,2}_{n}=\big(B^{3,2}_{n-1}*n\big)\cup \big((-B^{3,1}_{n-1})*(-n)\big)$. Since the reverse of a shelling order of a sphere is also a shelling order, it follows from the proof of the above theorem that a shelling order $\mathcal{O}_1$ of $-B^{3,1}_{n-1}$ extends to a shelling order $\mathcal{O}_2$ of $B^{3,2}_{n-1}$. Hence the shelling order of $B^{3,2}_{n-1}*n$ induced by $\mathcal{O}_2$ followed by the shelling order of $(-B^{3,1}_{n-1})*(-n)$ induced by $\mathcal{O}_1$ is a shelling order of $B^{4,2}_{n}$. Thus, the ball $B^{4,2}_{n}$ is shellable (for all $n\geq 5$); it is also $2$-stacked as all balls $B^{d,2}_n$ are. On the other hand, it was recently shown in \cite[Example 4]{Macchia-Wiebe} that $\partial B^{4,2}_{6}=\Delta^3_{6}$ is not polytopal.
\end{remark}

\section{Open problems}  \label{sec:open problems}
We close the paper with a few open problems. The result that $\Delta^3_{6}$ is not polytopal makes it very likely that the complexes $\Delta^3_{n}$ are not polytopal for all $n\geq 6$. It also begs us to ask the following question.

\begin{problem}\label{p1} According to a result of McMullen and Shephard \cite{McMShep}, for $d\geq 3$, a cs combinatorial $d$-sphere that is cs-$\lceil d/2\rceil$-neighborly and has more than $2(d+2)$ vertices cannot be realized as the boundary complex of a centrally symmetric polytope. Which of the cs spheres $\Delta^{d}_n$, where $d\geq 3$ and $n\geq d+2$, are realizable as the boundary complexes of non-cs polytopes? What about $\Lambda^{d}_n$?
\end{problem}

While this paper was under review, the first question in Problem \ref{p1} has been completely resolved by Pfeifle \cite{Pfeifle-20} (see also later work \cite{GouveiaMacchiaWiebe}) who proved that for all $d\geq 3$ and $n\geq d+2$ (including $n=d+2$), the complex $\Delta_n^{d}$ is {\em not} realizable as the boundary complex of a polytope. The second question remains open.

The rest of the problems concern the number of distinct combinatorial types of highly neighborly cs spheres.

\begin{problem} \label{Problem 1}
Let $k\geq 3$, and let $n$ be sufficiently large. Find $\Omega(2^n)$ pairwise non-isomorphic cs (combinatorial) $(2k-1)$-spheres with $2n$ vertices that are cs-$k$-neighborly. More optimistically, are there $2^{\Omega(n^k)}$ such spheres (for all $k\geq 2$)?
\end{problem}

\begin{problem} \label{Problem 2}
Let $\cs(d,n)$ denote the number of labeled cs $d$-spheres on $V_n$ and let $\ncs(d,n)$ denote the number of labeled cs $d$-spheres on $V_n$ that are cs-$\lceil d/2 \rceil$-neighborly. Is it true that  $$\lim_{n\to \infty} \frac{\cs(d,n)}{\ncs(d,n)} =1$$ for all odd $d\geq 3$? 
\end{problem}

Problem \ref{Problem 2} is analogous to Kalai's conjecture \cite[Section 6.3]{Kal} in the non-cs case. To start working on this problem, one may want to first establish some non-trivial upper and lower bound on $\cs(d, n)$. Part 1 of Problem~\ref{Problem 1} is motivated by our Theorem~\ref{thm: 2^n combinatorial types} while the more optimistic bound is motivated by a (non-cs) result of Nevo, Santos, and Wilson \cite{NeSanWil} along with Problem~\ref{Problem 2}. This result by Nevo, Santos, and Wilson asserts that for $k\geq 2$, there exist $2^{\Omega(n^k)}$ labeled triangulations of a $(2k-1)$-sphere with $n$ vertices. Since $n!=2^{O(n\log n)}$,  there are also $2^{\Omega(n^k)}$ pairwise non-isomorphic triangulations of a $(2k-1)$-sphere with $n$ vertices. 

As for establishing non-trivial lower bounds on $\cs(d,m)$, the results in this paper along with those in \cite{Kal} imply that  $\cs(2k-1,2n)\geq 2^{\Omega (n^{k-1})}$. Indeed, consider $\Lambda^{2k-1}_{2n-1}$, and let $\B$ be any of Kalai's squeezed spheres with at most $n$ vertices. Let $\rho: i \mapsto 2i+1$ be the map from the proof of Proposition~\ref{prop:squeezed-balls-as-subcomplexes}. Then $\rho(\B)$ is a combinatorial $(2k-1)$-ball that is a subcomplex of $(\Lambda^{2k-1}_{2n-1})_+$; in particular, $\rho(\B)$ and $-\rho(\B)$ share no common facets. Therefore, by replacing the subcomplexes  $\pm \rho(\B)$ of $\Lambda^{2k-1}_{2n-1}$ with $\pm \big(\partial\rho(\B) * (2n+2)\big)$, we obtain a new cs combinatorial $(2k-1)$-sphere, $\Lambda^{2k-1}_{2n-1}(\B)$. Furthermore, the resulting cs spheres are all distinct (as labeled spheres on $W_{2n}$). To see this, note that by \cite[Proposition 3.3]{Kal}, if the squeezed balls $\B_1$ and $\B_2$ are not equal, then $\partial\B_1\neq \partial \B_2$, and so $\lk\big(2n+2, \Lambda^{2k-1}_{2n-1}(\B_1)\big)\neq \lk\big(2n+2, \Lambda^{2k-1}_{2n-1}(\B_2)\big)$. We conclude that the number of cs combinatorial $(2k-1)$-spheres on the vertex set $W_{2n}$  is at least as large as the number of Kalai's squeezed $(2k-1)$-balls on $\leq n$ vertices. The promised lower bound on $\cs(2k-1,2n)$ then follows from \cite[Theorem 4.2]{Kal}.

{\small
	\bibliography{refs}
	\bibliographystyle{plain}
}
\end{document}